\documentclass[11pt]{amsart}
\usepackage[dvipsnames]{xcolor}
\usepackage{amssymb, amsmath, latexsym, amsthm, array, calrsfs, graphicx, hyperref, mathtools, multirow, latexsym, subfig, tikz-cd, verbatim, fge, xspace}

\usepackage{csquotes}
\usepackage{longtable}
\usepackage{mathpazo}
\usepackage{booktabs}
\DeclareMathAlphabet{\pazocal}{OMS}{zplm}{m}{n}
\usepackage{bm}

\newcommand{\C}{\mathbb{C}}
\newcommand{\F}{\mathbb{F}}

\newcommand{\N}{\mathbb{N}}
\newcommand{\PP}{\mathbb{P}}
\newcommand{\Q}{\mathbb{Q}}

\newcommand{\Z}{\mathbb{Z}}

\newcommand{\pA}{\pazocal{A}}
\newcommand{\pB}{\pazocal{B}}
\newcommand{\pC}{\pazocal{C}}

\newcommand{\pF}{\pazocal{F}}
\newcommand{\pG}{\pazocal{G}}
\newcommand{\pI}{\pazocal{I}}

\newcommand{\pM}{\pazocal{M}}

\newcommand{\fG}{\mathfrak{G}}
\newcommand{\fH}{\mathfrak{H}}

\newcommand{\fU}{\mathfrak{U}}

\newcommand{\sF}{\mathsf{F}}

\newcommand{\sL}{\mathsf{L}}
\newcommand{\sM}{\mathsf{M}}
\newcommand{\sN}{\mathsf{N}}

\newcommand{\sU}{\mathsf{U}}

\renewcommand{\setminus}{\mathbin{\fgebackslash}}

\newcommand{\Sn}[1]{S_{#1}}
\newcommand{\QSn}[1]{\mathfrak{S}_{#1}}
\newcommand{\QG}{\fG}
\newcommand{\subgroup}{\leq}

\newcommand{\End}[1]{\operatorname{End}(#1)}

\newcommand{\MatU}[2]{\sU_{#1,#2}} 

\newcommand{\Aut}{\operatorname{Aut}}
\newcommand{\cl}{\operatorname{cl}}

\newcommand{\Surj}{\operatorname{Surj}}
\newcommand{\Hom}{\operatorname{Hom}}
\newcommand{\Emb}{\operatorname{Emb}}

\newcommand{\PSL}{\operatorname{PSL}}

\newcommand{\rank}{\operatorname{rank}}

\newcommand{\Tup}[2]{\operatorname{Tup}_{#1}(#2)}

\newcommand{\QAut}[2]{\mathfrak{Aut}_{#1}(#2)}
\newcommand{\QAutBichon}[1]{\QAut{\textnormal{Bichon}}{#1}}
\newcommand{\QAutBanica}[1]{\QAut{\textnormal{Banica}}{#1}}

\newcommand{\tupI}{\overline{\pI}}
\newcommand{\tupB}{\overline{\pB}}
\newcommand{\tupF}{\overline{\pF}}
\newcommand{\tupC}{\overline{\pC}}

\DeclareMathOperator{\girth}{girth}

\newcommand{\tupInsert}[3]{#1 \cup_{#3} #2}
\newcommand{\freeprod}[3]{#1 \ast_{#3} #2}
\newcommand{\Oscar}{\texttt{OSCAR}\xspace}
\newcommand{\polymake}{\texttt{polymake}\xspace}
\newcommand{\Letterplace}{\texttt{Letterplace}\xspace}

\DeclareMathOperator{\com}{com}

\newtheorem{theorem}{Theorem}[section]
\newtheorem{lemma}[theorem]{Lemma}
\newtheorem{proposition}[theorem]{Proposition}

\newtheorem*{theorem*}{Theorem}

\theoremstyle{definition}
\newtheorem{algorithm}[theorem]{Algorithm}
\newtheorem{question}[theorem]{Question}

\newtheorem{example}[theorem]{Example}
\newtheorem{remark}[theorem]{Remark}
\newtheorem*{assumption*}{Assumption}

\theoremstyle{remark}

\newtheorem{definition}[theorem]{Definition}  

\numberwithin{equation}{section}
\numberwithin{table}{section}
\numberwithin{figure}{section}

\usepackage[colorinlistoftodos, bordercolor=orange, backgroundcolor=orange!20, linecolor=orange, textsize=scriptsize]{todonotes}
\begin{document} 

\title{Quantum automorphisms of matroids}

\author[D. Corey]{Daniel Corey}
\address{
University of Nevada, Las Vegas,
Department of Mathematical Sciences, 
Las Vegas, USA
}
\email{daniel.corey@unlv.edu}

\author[M. Joswig]{Michael Joswig}
\address{Technische Universit\"at Berlin, Chair of Discrete Mathematics / Geometry, Berlin, \& MPI Mathematics in the Sciences, Leipzig, Germany}
\email{joswig@math.tu-berlin.de}

\author[J. Schanz]{Julien Schanz}
\address{Universit\"at des Saarlandes, Fachbereich Mathematik,  Saarbr\"ucken, Germany}
\email{schanz@math.uni-sb.de}

\author[M. Wack]{Marcel Wack}
\address{
Technische Universit\"at Berlin, Chair of Discrete Mathematics / Geometry, Berlin}
\email{wack@math.tu-berlin.de}

\author[M. Weber]{Moritz Weber}
\address{Universit\"at des Saarlandes, Fachbereich Mathematik,  Saarbr\"ucken, Germany}
\email{weber@math.uni-sb.de}

\maketitle

\begin{abstract}
Motivated by the vast literature of quantum automorphism groups of graphs, we define and study quantum automorphism groups of matroids. A key feature of quantum groups is that there are many quantizations of a classical group, and this phenomenon manifests in the cryptomorphic characterizations of matroids. Our primary goals are to  understand, using theoretical and computational techniques,  the relationship between these quantum groups and to find when these quantum groups exhibit quantum symmetry. Finally, we prove a matroidal analog of Lov\'{a}sz's theorem characterizing graph isomorphisms in terms of homomorphism counts.

\medskip



\medskip


\end{abstract}

\section{Introduction}

In mathematics, symmetry is typically captured by groups and their actions. However, throughout the 20th century, it turned out that this class is too small to cover all aspects of symmetry. This led to  the birth of quantum groups in 1980s, due to the pioneering work of Drinfeld, Jimbo, Manin, Woronowicz and others, revealing a hidden---and often significantly richer---notion of symmetry: quantum symmetry. Since then, many instances of quantum symmetry have been investigated by numerous scientists with often strikingly deeper insights.

The mathematics behind quantum physics is grounded in the theories of $C^{\ast}$-algebras and von Neumann algebras as developed by Gelfand, Naimark, von Neumann and others in the 1930s and 1940s. Woronowicz \cite{Woronowicz} defines compact quantum groups in this setting, which provides the correct notion of quantum symmetry. From the work of Sh.\,Wang in the 1990s, we know that there are more ways of quantum permuting $N$ points rather than just to permute them. His free symmetric quantum group $\QSn{N}$ fits in the framework of Woronowicz, and it contains the well-known symmetric group $\Sn{N}$ \cite{Wang:1998}.  The containment is strict if and only if $N\geq 4$ \cite[p.\,496]{Wang:1999}. There is a precise definition of an action of a quantum group on $N$ points, and Wang showed that $\QSn{N}$ is the maximal compact quantum group admitting such an action. This yields a precise formulation of the statement that there are more quantum permutations than permutations. 

Note that $\QSn{N}$ qualifies as a quantum version of $\Sn{N}$ in the following sense. A compact quantum group $\QG$ (in the sense of Woronowicz) is a $C^{\ast}$-algebra (briefly, a complex Banach algebra endowed with a conjugate-linear involution compatible with the norm).  If $\QG$ is turned commutative (by dividing the commutator ideal), we obtain $C(G)$, the algebra of $\C$-valued continuous functions on some compact group $G$. In the case of $\QSn{N}$, this group is $\Sn{N}$.

That there is no unique quantum version of a given group is an interesting and fundamental feature of the subject.  In fact, there are groups with an abundance of quantizations. The group $U_N$ of unitary, complex $N\times N $ matrices has infinitely many quantizations. It is shown in  \cite{MangWeber} that to each subsemigroup $H\subseteq (\mathbb N_0,+)$ of the additive semigroup of natural numbers, there is a quantum group $\fU_N(H)$ such that
\begin{equation*}
U_N\leq \fU_N(H)\leq \fU_N.
\end{equation*}
Here, $\fU_N$ is Wang's free unitary quantum group, a kind of maximal quantization of $U_N$.

With the quantum symmetric group $\QSn{N}$ in hand, one can pass to quantum automorphism groups of finite spaces with additional structures. In 2003, Bichon defined quantum automorphism groups of finite graphs (without loops or  multiple edges) that is the terminal object in the category of quantum transformation groups of a graph $\Gamma$. In 2005, Banica gave a slightly modified definition. Both versions contain the (classical) automorphism group of the given graph  $\Gamma$, i.e., we have
\begin{equation}
\label{eq:quantumAutGraphs}
C(\Aut(\Gamma)) \leq \QAutBichon{\Gamma}\leq \QAutBanica{\Gamma}\leq  \QSn{N}
\end{equation}	
where $N$ is the number of vertices of $\Gamma$. Both quantum groups, Bichon's and Banica's, are quantum versions of $\Aut(\Gamma)$ in the above sense. It turns out that Banica's definition behaves nicer in many respects. For instance, we have $\QAut{\textnormal{Banica}}{\Gamma}=\QAut{\textnormal{Banica}}{\Gamma^c}$, where $\Gamma^c$ is the complement graph of $\Gamma$. Likewise, we have $\Aut(\Gamma)=\Aut(\Gamma^c)$, but there are graphs for which $\QAut{\textnormal{Bichon}}{\Gamma} \neq \QAut{\textnormal{Bichon}}{\Gamma^c}$ holds (e.g., the complete graph on $N\geq 4$ vertices). In any case, for the moment, we know of two different definitions for quantum automorphism groups of graphs.

The graph $\Gamma$ has \textit{quantum symmetry} if 
\begin{equation*}	
C(\Aut(\Gamma)) \neq \QAutBanica{\Gamma}.
\end{equation*}
In principle, one could call it \enquote{quantum symmetry in Banica's sense} and also define a notion of \enquote{quantum symmetry in Bichon's sense}, but the latter has not appeared in the literature so far. For each individual inclusion in \eqref{eq:quantumAutGraphs} there are examples of graphs such that it is strict or an equality, and this provides a certain hierarchy of quantum symmetry for graphs.

The study of quantum automorphism groups of graphs has recently experienced quite some interest, in particular due to the links to algebraic combinatorics (in terms of a \enquote{quantum Lov\'{a}sz Theorem}, see \S\ref{sec:Lovasz}) and quantum information theory, see the work of Mancinska and Roberson \cite{MancinskaRoberson}.  There have been several attempts to generalize the theory to other classes of combinatorial objects such as hypergraphs, Hadamard matrices \cite{Gromada} and quantum graphs \cite{Daws:2022}.

In this work we propose several definitions for quantum automorphism groups of matroids. Introduced by Whitney \cite{Whitney:1935} in the 1930's, a matroid is a common combinatorial generalization of linear dependence of vectors and cycles in a graph. Matroids are notorious for their cryptomorphic definitions; we focus on the characterizations involving independent sets, bases, flats, and circuits. Each axiom system determines the same classical automorphism group of a matroid $\sM$.
We define quantum automorphism groups of the matroid $\sM$
\begin{equation*}
  \QAut{\pI}{\sM},\ \QAut{\pB}{\sM},\ \QAut{\pF}{\sM},\ \QAut{\pC}{\sM}
\end{equation*}
by adding relations to those of $\QSn{N}$ derived from the independent sets, bases, flats, and circuits definitions of $\sM$, respectively, in the spirit of Bichon and Banica.  We have the surprising feature that each of these yields an, a priori, different definition of a quantum automorphism group. However, for nice classes of matroids, we have a chain of inclusions, similar to the case of Bichon's and Banica's definitions of quantum automorphism groups of graphs. Our main results are summarized in the following theorem; see \S\ref{sec:QuantumAutomorphismsAxioms} for details on the terminology.  
\begin{theorem*}
  For every  matroid $\sM$ we have 
\begin{equation*}
  C(\Aut(\sM)) = \QAut{\pF}{\sM} \leq \QAut{\pB}{\sM} = \QAut{\pI}{\sM}.
\end{equation*}
If $\sM$ is a simple rank $3$ matroid and the ground set $E(\sM)$ is not equal to $F_1 \cup F_2 \cup F_3$ for triangles  $\{F_1, F_2, F_3\}$, then
\begin{equation*}
  C(\Aut(\sM)) = \QAut{\pF}{\sM} \leq \QAut{\pC}{\sM} \leq \QAut{\pB}{\sM} = \QAut{\pI}{\sM}.
  \end{equation*}
\end{theorem*}
Using the free open-source computer algebra system \Oscar \cite{OSCAR-book, Oscar}, which is written in \texttt{Julia}, we develop code to compute $\QAut{\pB}{\sM}$ and $\QAut{\pC}{\sM}$, and and determine whether these are commutative. In particular, we find several examples where $\QAut{\pB}{\sM} \leq \QAut{\pC}{\sM}$ and others where $\QAut{\pC}{\sM} \leq \QAut{\pB}{\sM}$. 

In \S\ref{sec:Lovasz},  we derive a matroidal analog of Lov\'{a}sz's graph homomorphism count theorem.  This theorem asserts that two graphs are isomorphic if and only if the number of graph homomorphism from $\Gamma$ \textit{to} these graphs are equal for all graphs $\Gamma$.  By the main result of \cite{MancinskaRoberson}, two graphs are quantum isomorphic, in the sense of Banica, if these homomorphism counts agree for all \textit{planar} graphs. We take the first steps in investigating the connection between (quantum) isomorphic matroids and counts of matroid maps (i.e., \textit{strong maps}) by proving an analog of Lov\'{a}sz's theorem for matroids. The crucial difference is that matroid isomorphism is determined by strong map counts \textit{from} the candidate matroids. 

Our work is just the beginning of the study of quantum symmetry for matroids. Our hope is to provide insight into the world of matroids by dividing the class of matroids into subclasses of ``more complicated/richer'' matroids (having a high degree of quantum symmetry) and \enquote{easier} ones (with a low or no degree of quantum symmetry).

\subsection*{Code} The code used to collect the data recorded in Appendix \ref{sec:tables} may be found at
\begin{center}
	\url{https://github.com/dmg-lab/QuantumAutomorphismGroups.jl} \enspace .
\end{center}
This \texttt{GitHub} repository also contains noncommutative Gr\"{o}bner bases for the quantum automorphism groups obtained, stored in the \texttt{mrdi} file format \cite{mrdi-files}. 

\subsection*{Conventions and notation} Denote by $\N_{0}$ the set of nonnegative integers. Ordinary groups are written in normal font and quantum groups are written in fraktur. For example, the symmetric group on the (usually finite) set $E$ is denoted by $\Sn{E}$, whereas the quantum symmetric group is denoted by $\QSn{E}$.  All graphs are simple in the sense that they do not have loops or multiple elements. All ideals of noncommutative rings are understood to be two-sided.  

\subsection*{Acknowledgements}
We thank Viktor Levandovskyy and Simon Schmidt for helpful discussions.  Corey, Joswig, Schanz and Weber are partially supported by the Deutsche Forschungsgemeinschaft (DFG, German Research Foundation), SFB-TRR 195-286237555 ``Symbolic Tools in Mathematics and their Application.'' Joswig has further been supported by DFG under Germany's Excellence Strategy -- The Berlin Mathematics Research Center MATH$^+$ (EXC-2046/1, project ID 390685689). 

\section{Matroids and their automorphism groups}

Matroids admit numerous cryptomorphic axiomatic systems. We recommend to the reader the textbook \cite{Oxley} as a general reference for matroid theory. From the viewpoint of bases, a matroid is a pair $\sM = (E,\pB)$ where $E$ is a nonempty finite set, called the \textit{ground set} of $\sM$, and $\pB$ is a nonempty subset of $2^E$ that satisfies the \textit{basis exchange axiom}: for each pair of distinct elements $A,B\in \pB$ and $a\in A\setminus B$, there is a $b\in B\setminus A$ such that $A\setminus \{a\} \cup \{b\}$ is in $\pB$.  An element of $\pB$ is called a \textit{basis} of $\sM$. To emphasize the dependence on $\sM$, we write $E(\sM)$ for the ground set of $\sM$ and  $\pB(\sM)$ for the bases of $\sM$.   It follows readily from this axiom that all elements of $\pB(\sM)$ have the same size $r$, and this is called the \textit{rank} of $\sM$. The other common characterizations of $\sM$ may be derived from $\pB(\sM)$ in the following way. 

\begin{itemize}
\item  An \textit{independent set} of $\sM$ is a subset $A\subseteq E(\sM)$ that is contained in a basis, and the set of independent sets is denoted by $\pI(\sM)$.
\item  The \textit{rank function} of $\sM$ is $\rank_{\sM}:2^{E(\sM)} \to \N_{0}$ defined by 
  \begin{equation*}
    \rank_{\sM}(A) = \max(|B| \, : \, B\in \pI(\sM)).  		
  \end{equation*}
\item A \textit{flat} of $\sM$ is a subset $A \subseteq E(\sM)$ such that $\rank_{\sM}(A\cup \{b\}) > \rank_{\sM}(A)$ for all $b\notin A$. The set of flats of $\sM$ is denoted by $\pF(\sM)$. 
\item A \textit{circuit} of $\sM$ is a subset $A \subseteq E(\sM)$ such that $A\setminus \{a\}$ is independent for all $a\in A$. The set of circuits of $\sM$ is denoted by $\pC(\sM)$. 
\end{itemize}

A \textit{loop} of $\sM$ is an element $x\in E(\sM)$ such that $\{x\}$ is dependent. Two non-loop elements $x,y\in E(\sM)$ are \textit{parallel} if $\{x,y\}$ is dependent. The matroid $\sM$ is \textit{simple} if it does not have loops or parallel elements. 

Since matroids are a combinatorial abstraction of dependence coming from linear algebra and graph theory, it is useful to understand how the above terms are interpreted in these contexts. The linear-algebraic model for a matroid is derived from a list of vectors $X = (x_{i} \, : \, i\in E)$, indexed by a finite set $E$, that span a $r$-dimensional vector space $V$. The matroid of $X$, denoted $\sM[X]$, is the matroid whose ground set is $E$ and 
\begin{equation*}
  \pB(\sM[X]) = \{ B\subseteq E \, : \, (x_{i} \, : \, i\in B) \text{ is a vector space basis of } V \}. 	
\end{equation*}
The independent sets of $\sM[X]$ correspond to collections of linearly independent vectors in $X$, the rank function of $\sM[X]$ records the dimension of the linear span of the input vectors, and the flats of $\sM[X]$ are the subsets of $E$ such that the linear span of the corresponding vectors contains no other elements of $X$. Given a field $\F$, a matroid $\sM$ is $\F$-realizable if $\sM = \sM[X]$ for a $E(\sM)$-indexed list $X$ of vectors in a $\F$-vector space. When $x_i\neq 0$ for each $i\in E(\sM)$, the vectors $x_i$ define points in the projective space $\PP(V)$; this is called a \textit{projective realization} of $\sM[X]$. 

Circuits are best understood from the graph-theoretic viewpoint. Let $\Gamma$ be a finite graph with vertex set $V(\Gamma)$ and edge set $E(\Gamma)$. The matroid of $\Gamma$, denoted $\sM[\Gamma]$, is the matroid with ground set $E = E(\Gamma)$ and the bases $\pB(\sM[\Gamma])$ consist of the spanning forests of $\Gamma$. A circuit of $\sM[\Gamma]$ is subset of edges that form a cycle of $\Gamma$. 

Given two matroids $\sM_1$ and $\sM_2$, an \textit{isomorphism} of $\sM_1$ and $\sM_2$ is a bijection $\varphi: E(\sM_1) \to E(\sM_2)$ such that $A\in \pB(\sM_1)$ if and only if $\varphi(A)\in \pB(\sM_2)$. An \textit{automorphism} of a matroid $\sM$ is an isomorphism from $\sM$ to itself. The group (under function composition) of all automorphisms of $\sM$ is called the \textit{automorphism group} of $\sM$ and is denoted by $\Aut(\sM)$. 

\begin{example}[Uniform matroid]
  The rank-$r$ \textit{uniform matroid} on the finite set $E$, denoted $\MatU{r}{E}$ is the matroid with ground set $E$ and whose bases are the the $r$-element subsets of $E$. When $E = \{1,\ldots,n\}$, we simply write $\MatU{r}{n}$. This is the matroid of a collection of $n$ vectors in $V$ in linear general position, where $V$ is a $r$-dimensional vector space over a field of infinite order. The automorphism group of $\MatU{r}{E}$ is the symmetric group $\Sn{E}$.
\end{example}

\begin{example}[Fano matroid]
\label{ex:fano}
A fundamental example is the Fano matroid defined by $\sF = \sM[X]$ where $X$ is the sequence of vectors in $(\F_{2})^{3}$ given by the columns of the matrix (which we also denote by $X$)
\begin{equation*}
  X =
  \begin{bmatrix}
    1 & 0 & 1 & 0 & 1 & 0 & 1 \\ 
    0 & 1 & 1 & 0 & 0 & 1 & 1 \\
    0 & 0 & 0 & 1 & 1 & 1 & 1
  \end{bmatrix}.
\end{equation*}
These are exactly the seven nonzero vectors in $(\F_{2})^{3}$. The automorphism group of $\sF$ is isomorphic to $\PSL_3(\F_2)$. This matroid is significant because it is the smallest example which is not $\C$-realizable. 
\end{example}

\begin{figure}
	\includegraphics[height=2.5cm]{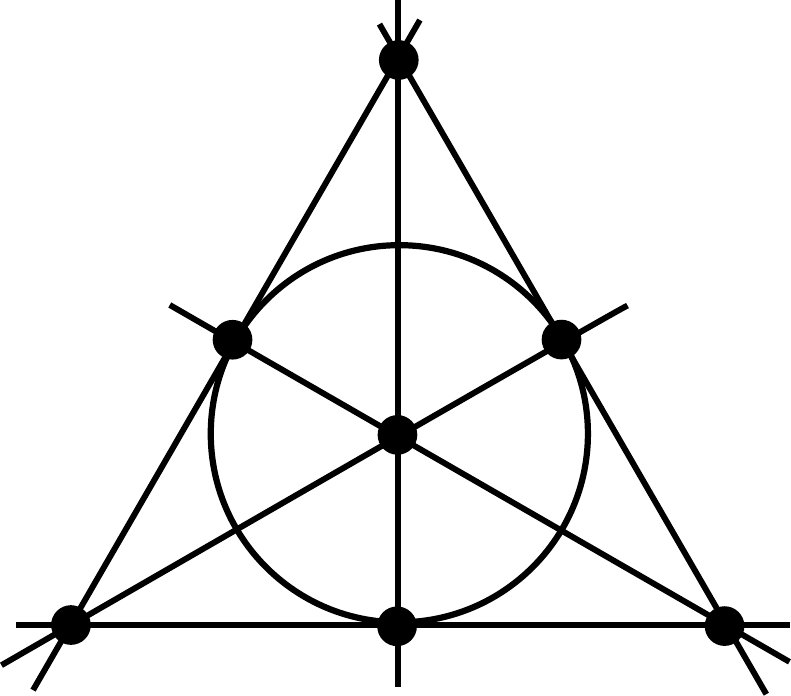}
	\caption{A projective realization over $\F_2$ of the Fano matroid $\sF$}
	\label{fig:FanoMatroid}
\end{figure}

The automorphism group of a graph $\Gamma$, denoted by $\Aut(\Gamma)$,  differs from $\Aut(\sM[\Gamma])$ in significant ways. The former is a subgroup of $\Sn{V(\Gamma)}$, whereas the latter is a subgroup of $\Sn{E(\Gamma)}$. To provide a direct comparison, we may consider the action of $\Aut(\Gamma)$ on $E(\Gamma)$, which defines a group homomorphism $\Aut(\Gamma) \to \Sn{E(\Gamma)}$. The image, denoted by $\Aut_{E}(\Gamma)$, is a subgroup of $\Sn{E(\Gamma)}$.  Nevertheless, $\Aut_E(\Gamma)$ and $\Aut(\sM[\Gamma])$ need not be the same subgroup of $\Sn{E(\Gamma)}$.  The difference between these two groups is the content of  Whitney's $2$-isomorphism theorem \cite{Whitney1933}; see also \cite[Theorem 5.3.1]{Oxley}.

\section{Quantum permutation groups}
\label{sec:quantumPermutationGroups}

An \textit{involutive algebra} is a complex algebra $\QG$ that has an conjugate-linear involution $u\mapsto u^{\ast}$ satisfying $(uv)^{\ast} = v^{\ast} u^{\ast}$ for all $u,v \in \QG$.  All quantum groups we consider are subgroups of the quantum symmetric group as defined in \cite{Wang:1998}, which we now describe. Let $E$ be a finite set. Define the noncommutative polynomial ring in $|E|^2$ variables
\begin{equation*}
\C\langle E^2 \rangle = \C\langle u_{ij} \, : \, i, j \in E \rangle.
\end{equation*}
The ring $\C\langle E^2 \rangle$ is an involutive algebra whose involution is induced by $(u_{ij})^{\ast} = u_{ij}$.  A subset $R\subseteq \C\langle E^2 \rangle$ is \textit{self-adjoint} if $R^{\ast} = R$, where 
\begin{equation*}
	R^{\ast} = \{f^{\ast} \, : \, f\in R \}.
\end{equation*}
If $I$ is a self-adjoint (two-sided) ideal, then $\C\langle E^2 \rangle / I$ is an involutive algebra with involution induced by the one on $\C\langle E^2\rangle$.
Define the ideal $I_E \subset \C\langle E^2 \rangle$ by
\begin{equation}\label{eq:IE}
  I_E = \left\langle
    \begin{array}{c}
      u_{ij}^2 - u_{ij}; \hspace{5pt} u_{ik}u_{i\ell},\hspace{3pt} u_{kj}u_{\ell j} \; (k\neq \ell);  \hspace{5pt} \sum_{k\in E} u_{kj} - 1,\; \sum_{k\in E} u_{ik} - 1
    \end{array} \right\rangle.
\end{equation}
 The \textit{quantum symmetric group} on $E$ is 
\begin{equation} \label{eq:QSn}
  \QSn{E} = \C\langle E^2 \rangle / I_E
\end{equation}
equipped with the coproduct
\begin{equation}\label{eq:Delta}
  \Delta: \QSn{E} \to \QSn{E} \otimes \QSn{E} \hspace{20pt} \Delta(u_{ij}) = \sum_{k\in E} u_{ik} \otimes u_{kj}.  
\end{equation}
The ideal $I_E$ is self-adjoint and so $\QSn{E}$ is an involutive algebra.
Unless stated otherwise tensor products are taken over $\C$.
A \textit{quantum permutation group} $\QG$ on $E$ is an involutive algebra
\begin{equation} \label{eq:QG}
  \QG = \C\langle E^2 \rangle / I
\end{equation}
where $I\supseteq I_E$ is a self-adjoint ideal and the coproduct $\Delta$ restricts to a coproduct on $\QG$. To emphasize the dependence on $\QG$, we write $I(\QG)$ for its defining ideal.  Given quantum permutation groups $\QG_{1}$ and $\QG_{2}$ on $E$, we have that $\QG_1$ is a \textit{subgroup} of $\QG_2$, written $\QG_{1} \leq \QG_{2}$ if $I(\QG_{1}) \supseteq I(\QG_{2})$.

Given a finite group   $G$, denote by $C(G)$ the involutive algebra of complex-valued functions on $G$. A quantum permutation group $\QG\leq \QSn{E}$ is \textit{commutative} if $u_{ij}u_{k\ell} = u_{k\ell} u_{ij}$ for all $i,j,k,\ell \in E$. By the Gelfand--Naimark theorem, the quantum permutation group $\QG$ is commutative if and only if $\QG = C(G)$ for a finite group $G$, see \cite[Exercise~1.10]{Freslon}. We use this fact in an essential way in Proposition \ref{prop:recoverOrdinaryAut}.

Let $\QG \leq \QSn{E}$ and $\fH\leq \QSn{F}$ be quantum permutation groups. Their \textit{free product}, denoted by $\freeprod{\QG}{\fH}{}$, is the subgroup of $\QSn{E\sqcup F}$ whose ideal is 
\begin{equation}
  I(\freeprod{\QG}{\fH}{}) = I(\QG) + I(\fH) + \langle u_{ij} \, : \, (i\in E \text{ and } j\notin F) \text{ or }  (i\notin E \text{ and } j\in F)\rangle. 
\end{equation} 

\begin{remark}\label{rem:coefficients}
  Instead of the noncommutative polynomial ring $\C\langle E^2 \rangle$ with complex coefficients, we could also consider its analog $\Z\langle E^2 \rangle$ over the integers.
  Since the generators of the ideal $I_E$ from \eqref{eq:IE} have integer coefficients, the quotient $\Z\langle E^2 \rangle / I_E$ is defined.
  The coproduct formula \eqref{eq:Delta} remains valid, and so that quotient is an involutive algebra.
  Now, if the ideal $I$ in \eqref{eq:QG} makes sense over the integers, too, we obtain the \textit{integral quantum permutation group} $\QG_\Z=\Z\langle E^2 \rangle /  I$.
  Reading the additive group of the algebra $\QG_\Z$ as a $\Z$-module then allows for tensoring over $\Z$.
  It follows that $\QG_\Z\otimes\C$ is isomorphic to $\QG$ as involutive algebras.
  This is useful for our computations in \S\ref{sec:computations}, which employ $\QG_\Z \otimes \Q$.
\end{remark}

\section{Quantum automorphisms associated to different axiom systems}
\label{sec:QuantumAutomorphismsAxioms}

We begin with a general procedure for producing subgroups of $\QSn{E}$, where $E$ is a finite set. Denote by $\Tup{}{E}$ the set of tuples of elements of $E$ that have length at most $|E|$.   Given tuples $A = (a_1,\ldots,a_k)$ and $B = (b_1,\ldots,b_k)$ of the same length, define
\begin{equation*}
	u_{AB} = u_{a_1, b_1} \cdots u_{a_k, b_k}. 
\end{equation*}
We gather some useful properties of the terms $ u_{AB}$ in $\QSn{E}$.
\begin{proposition}
\label{prop:tupleCoproduct}
Suppose $A, B\in E^{k}$. The following equalities hold in $\QSn{E}$. 
\begin{enumerate}
	\item We have
	\begin{equation*}
		\sum_{C\in E^k} u_{AC} = 1 \hspace{5pt} \text{ and } \hspace{5pt} \sum_{C\in E^k} u_{CB} = 1.
	\end{equation*}
	\item The coproduct of $u_{AB}$ is
	\begin{equation*}
		\Delta(u_{AB}) = \sum_{C\in E^k} u_{AC} \otimes u_{CB}
	\end{equation*}
\end{enumerate}
\end{proposition}

\begin{proof}
	To prove (1), by symmetry, it suffices to prove just the left equality. We proceed by induction on $k$. The base case $k=1$ follows from the definition of $\QSn{E}$.  Let $A = (a_1,\ldots,a_{k-1}, a)$ and $A' = (a_1,\ldots,a_{k-1})$. Then	
	\begin{equation*}
		\sum_{C\in E^k} u_{AC} = \sum_{C'\in E^{k-1}} u_{A'C'} \sum_{c\in E} u_{ac}
	\end{equation*}
	Both sums on the right equal 1, the former by the inductive hypothesis and the latter by the base case. 
	
	For (2), we again proceed by induction on $k$. The base case $k=1$ is just the coproduct formula. Let $A,A'$ be as in the proof of (1) and let $B = (b_1,\ldots, b_{k-1}, b)$ and $B' = (b_1,\ldots, b_{k-1})$. Then
	\begin{equation*}
		\Delta(u_{AB}) = \Delta(u_{A'B'}) \Delta(u_{ab}) = \left(\sum_{C'\in E^{k-1}} u_{A'C'} \otimes u_{C'B'}\right) \left(\sum_{c\in E} u_{ac}\otimes u_{cb} \right) = \sum_{C\in E^k} u_{AC} \otimes u_{CB}. 
	\end{equation*}
	The second equality follows from the inductive hypothesis. 
\end{proof}

\noindent Fix a nonempty subset $\pA\subseteq \Tup{}{E}$.
Define the ideal $I_{\pA}$ of $\QSn{E}$ by 
\begin{equation}\label{eq:IA}
  I_{\pA} = \langle u_{AB} \, : \, (A\in \pA \text{ and } B\notin \pA) \text{ or } (A\notin \pA \text{ and } B\in \pA ) \rangle. 
\end{equation}
Let $\QG_{\pA} = \QSn{E}/I_{\pA}$. 
\begin{proposition}
\label{prop:subgroupQSn}
The quotient $\QG_{\pA}$ is a subgroup of the quantum symmetric group $\QSn{E}$.
\end{proposition}

\begin{proof}
	The ideal $I_{\pA}$ is self-adjoint, and therefore $\QG_{\pA}$ is an involutive algebra. It remains to show that the coproduct on $\Delta$ on $\QSn{E}$ restricts to a coproduct on $\QG_{\pA}$. By Proposition \ref{prop:tupleCoproduct} we have
	\begin{equation*}
		\Delta(u_{AB}) = \sum_{C\in E^k} u_{AC} \otimes u_{CB}.
	\end{equation*}
	Suppose $u_{AB} \in I_{\pA}$.   Without loss of generality, we may assume that $A\in \pA$ and $B \notin \pA$. If $C\in \pA$, then $u_{CB} \in I_{\pA}$. Otherwise, $C\notin \pA$, and so $u_{AC} \in I_{\pA}$. Therefore, each summand of $\Delta(u_{AB})$ lies in $I_{\pA} \otimes I_{\pA}$, as required.
\end{proof}

Let $\sM$ be a rank-$r$ matroid, and let $A$ be a tuple. Then $A$ is:
\begin{itemize}
	\item an \textit{independent tuple} if it has no repeated elements and its underlying set is independent;
	\item a \textit{basis tuple} if it is independent and has length $r$;
	\item a \textit{flat tuple} if it has no repeated elements and its underlying set is a flat of $\sM$;
	\item a \textit{circuit tuple} if $A = (a,a)$ for a nonloop $a$ or if $A$ has no repeating elements and its underlying set is a circuit of $\sM$.
\end{itemize}
Denote by $\tupI(\sM)$, $\tupB(\sM)$, $\tupF(\sM)$, and $\tupC(\sM)$ the sets of independent, basis, flat, and circuit tuples of $\sM$, respectively. We may apply Proposition \ref{prop:subgroupQSn} to these sets to define quantum automorphism groups of matroids. 
\begin{definition}
\label{def:quantumAutMatroid}
Let $\sM$ be a matroid. 
\begin{itemize}
	\item The \textit{independent sets} quantum automorphism group is $\QAut{\pI}{\sM} = \QG_{\tupI(\sM)}$.  
	\item The \textit{bases} quantum automorphism group is $\QAut{\pB}{\sM} = \QG_{\tupB(\sM)}$.  
	\item The \textit{flats} quantum automorphism group is $\QAut{\pF}{\sM} = \QG_{\tupF(\sM)}$.  
	\item The \textit{circuits} quantum automorphism group is $\QAut{\pC}{\sM} = \QG_{\tupC(\sM)}$.  
\end{itemize}
\end{definition}
\noindent Throughout the paper, we follow the convention described in \S\ref{sec:quantumPermutationGroups} when denoting the ideals of these quantum automorphism groups, e.g., we write $I(\QAut{\pI}{\sM})$ instead of $I_{\tupI(\sM)}$. 

Given a quantum permutation group $\QG \leq \QSn{E}$, denote by $\QG^{\com}$ the commutative quantum permutation group 
\begin{equation*}
	\QG^{\com} = \QG / \langle u_{ab}u_{cd} - u_{cd}u_{ab} \, : \, a,b,c,d \in E\rangle.
\end{equation*}
The following proposition verifies that the quantum automorphism groups in Definition \ref{def:quantumAutMatroid} are, a priori, quantizations of $\Aut(\sM)$.

\begin{proposition}
\label{prop:recoverOrdinaryAut}
	The commutative quantum groups
	\begin{equation*}
		\QAut{\pI}{\sM}^{\com},\; \QAut{\pB}{\sM}^{\com}, \;\QAut{\pF}{\sM}^{\com},\; \QAut{\pC}{\sM}^{\com}
	\end{equation*}
	  are all isomorphic to $C(\Aut(\sM))$. 
\end{proposition}

Before proving this proposition, we establish some notation. Given a finite set $E$, denote by $\End{\C^{E}}$ the ring of linear maps $\C^{E} \to \C^{E}$. We view the symmetric group $\Sn{E}$ as the multiplicative subgroup of permutation matrices in $\End{\C^{E}}$. Given $\sigma\in \Sn{E}$ and $i,j\in E$, denote by $\sigma_{ij}$ the $(i,j)$--entry of the permutation matrix associated to $\sigma$, and define the homomorphism $\varphi_{\sigma}:\QSn{E} \to \C$ by $\varphi_{\sigma}(u_{ij}) = \sigma_{ij}$.

\begin{proof}[Proof of Proposition \ref{prop:recoverOrdinaryAut}]
We focus on the case $\QAut{\pI}{\sM}$, as the others are similar. Let $\QG = \QAut{\pI}{\sM}^{\com}$. By \cite[Exercise~1.10]{Freslon}, $\QG \cong C(G)$ where $G\leq \Sn{E}$ is the permutation group
\begin{equation*}
	G = \{\sigma \in \Sn{E} \, : \, \varphi_{\sigma}(f) = 0 \text{ for all  } f\in I(\QG)\}.
\end{equation*} 
We claim that $G = \Aut(\sM)$. 
Given tuples 
$A = (a_1,\ldots,a_k)$ and $B=(b_1,\ldots,b_k)$,  we have
\begin{equation}
\label{eq:varphisigma}
	\varphi_{\sigma}(u_{AB}) = \sigma_{a_1b_1} \cdots \sigma_{a_kb_k} = 
	\begin{cases}
		1 & \text{if } \sigma(A) = B, \\
		0 & \text{if } \sigma(A) \neq B.
	\end{cases}
\end{equation}
Suppose $\sigma \in \Aut(\sM)$ and $f\in I(\QG)$---without loss of generality we may assume that $f = u_{AB}$ where  $A$ is independent and $B$ is dependent. Since $\sigma$ is an automorphism of $\sM$, $\sigma(A)$ cannot equal  $B$, and so $\varphi_{\sigma}(u_{AB}) = 0$ by Equation  \eqref{eq:varphisigma}. Conversely, suppose $\sigma \in G$. If $A$ is independent, then $\sigma(A) \neq B$ for all dependent tuples $B$ by Equation \eqref{eq:varphisigma} and the definition of $I(\QG)$, and so $\sigma(A)$ is independent. Similarly, if $D = \sigma(A)$ is independent, then $\sigma^{-1}(D)\neq C$ for all dependent tuples $C$ (since $\sigma^{-1}\in G$), and so $A$ is independent. Therefore $\sigma\in \Aut(\sM)$, as required. 
\end{proof}

For the remainder of this section, we describe the relationship between the various quantum automorphism groups of matroids. 

\begin{theorem}
	\label{thm:QAutIeqQAutB}
	For  any matroid $\sM$, 
	\begin{equation*}
		\QAut{\pI}{\sM} = \QAut{\pB}{\sM}.
	\end{equation*}
\end{theorem}

\noindent  We begin with the following technical lemma. Given a tuple $A$ and an element $x$, the tuple $\tupInsert{A}{x}{i}$ is obtained by inserting $x$ at position $i$. 

\begin{lemma}
\label{lem:inclusion-meta}
	Let $\pA_1, \pA_2\subseteq \Tup{}{E}$. Suppose that, for each pair of tuples $A = (a_1,\ldots,a_{\ell})$, $B = (b_1,\ldots,b_{\ell})$ of the same length with $A\in \pA_2$ and $B\notin \pA_2$, there is a $x\in E$ and $i\in \{1,\ldots,\ell\}$ such that $\tupInsert{A}{x}{i} \in \pA_1$,  but $\tupInsert{B}{y}{i} \notin \pA_1$ 	for all $y\in E$. Then $\QG_{\pA_1} \subgroup \QG_{\pA_2}$. 
\end{lemma}

\begin{proof}
We must show that $I_{\pA_2} \subseteq I_{\pA_1}$. Suppose $u_{AB} \in I_{\pA_2}$, without loss of generality, we may assume that 
\begin{equation*}
	A = (a_1,\ldots,a_{\ell}) \in \pA_2 \hspace{5pt} \text{ and } \hspace{5pt} B = (b_1,\ldots,b_{\ell}) \notin \pA_2.
\end{equation*}
Let $x\in E$ and  $1\leq i\leq \ell$ be as in the lemma, and set $A' = (a_1,\ldots, a_{i-1})$, $A'' = (a_{i+1},\ldots, a_{\ell})$, $B' = (b_1,\ldots, b_{i-1})$, and $B'' = (b_{i+1},\ldots, b_{\ell})$. By the relation $\sum_{y\in E} u_{xy} =1$, we have
\begin{equation*}
	u_{AB} = \sum_{y\in E} u_{A'B'} u_{xy} u_{A''B''}
\end{equation*}
By hypothesis, each term in this sum lies in $I_{\pA_1}$, and therefore so does $u_{AB}$, as required. 
\end{proof}

\begin{proof}[Proof of Theorem \ref{thm:QAutIeqQAutB}]
For $1 \leq k < \rank(\sM)$, denote by $\tupI_k$ the set of independent tuples of length $k$. Suppose  $A\in \tupI_{k}(\sM)$ and $B\in E(\sM)^{k}\setminus \tupI_{k}$. Because $A$ is independent and not a basis, there is a $x\in E(\sM) \setminus A$ such that the tuple  $\tupInsert{A}{x}{k}$ is independent. Furthermore, because $B$ is dependent, the tuple $\tupInsert{B}{y}{k}$ is also dependent. By Lemma \ref{lem:inclusion-meta}, we have that $\QG_{\tupI_{k+1}} \leq \QG_{\tupI_{k}}$.   Thus, we get chain
\begin{equation*}
	 \QAut{\pB}{\sM} = \QG_{\tupI_r} \leq  \QG_{\tupI_{r-1}} \leq \cdots \leq \QG_{\tupI_1}. 
\end{equation*}
where $r=\rank(\sM)$, and hence $\QAut{\pB}{\sM} \subgroup \QAut{\pI}{\sM}$.  The inclusion $\QAut{\pI}{\sM} \subgroup \QAut{\pB}{\sM}$ follows from $\tupB(\sM) \subseteq \tupI(\sM)$.
\end{proof}

Next, we compare $\QAut{\pB}{\sM}$ and $\QAut{\pC}{\sM}$ for simple matroids $\sM$ (i.e., without loops or parallel elements) with rank 3.  We define a \textit{triangle} to be a triple of three rank-2 flats $\{F_1,F_2,F_3\}$ that pairwise intersect but $F_1 \cap F_2 \cap F_3 = \emptyset$.

\begin{theorem}
	\label{thm:QAutCInQAutB}
	Suppose $\sM$ is a simple, rank $3$ matroid. If $E(\sM)\neq F_1 \cup F_2 \cup F_3$ for all triangles $\{F_1,F_2,F_3\}$, then
	\begin{equation*}
	\QAut{\pC}{\sM} \leq \QAut{\pB}{\sM}.
	\end{equation*} 
\end{theorem}

\noindent In the proof of this theorem, we make use of the \textit{closure} operator of a matroid $\sM$.  This is the function $\cl_{\sM}:2^{E(\sM)} \to \pF(\sM)$ defined by
\begin{equation*}
	\cl_{\sM}(A) = \{x\in E(\sM) \, : \, \rank_{\sM}(A\cup\{x\}) = \rank_{\sM}(A)\}.
\end{equation*}

\begin{proof}[Proof of Theorem \ref{thm:QAutCInQAutB}]
	We must show that $I(\QAut{\pB}{\sM}) \subseteq I(\QAut{\pC}{\sM})$. Suppose $A=(a_1,a_2,a_3)$ is a basis tuple and $B = (b_1,b_2,b_3)$ is dependent. If $B$ is a circuit tuple, then $u_{AB} \in I(\QAut{\pC}{\sM})$. Suppose that $B$ is not a circuit. Because $\sM$ is simple, the tuple $B$ must have repeated elements. If $b_1 = b_2$ or $b_2 = b_3$, then $u_{AB} \in I(\QSn{E}) \subseteq I(\QAut{\pC}{\sM})$. Finally, consider the case $b_1=b_3$. Let $F_1 = \cl(\{a_2,a_3\})$, $F_2 = \cl(\{a_1,a_3\})$, and $F_3 = \cl(\{a_1,a_2\})$. By hypothesis, there is a $a_4\in E(\sM) \setminus (F_1\cup F_2 \cup F_3)$, and therefore $(a_1,a_2,a_3,a_4)$ is a circuit tuple. So
	\begin{equation*}
		u_{AB} = u_{a_1b_1} u_{a_2b_2} u_{a_3b_1} = \sum_{x\in E} u_{a_1b_1} u_{a_2b_2} u_{a_3b_1} u_{a_4x}. 
	\end{equation*}
	Because $(b_1,b_2,b_1,x)$ is not a circuit tuple for each $x\in E$,  each summand on the right is in  $I(\QAut{\pC}{\sM})$, as required. 
\end{proof}

\begin{theorem}
\label{thm:flatsTrivial}
	For any matroid $\sM$, we have
	\begin{equation*}
		\QAut{\pF}{\sM} = C(\Aut(\sM)). 
	\end{equation*}
\end{theorem}

\noindent Suppose $N = |E(\sM)|$. Because $u_{AB} \in I(\QAut{\pF}{\sM})$ for $A,B\in E(\sM)^{N}$ such that $A$ does not have repeated elements and $B$ has repeated elements, this theorem is a direct consequence of Proposition \ref{prop:recoverOrdinaryAut} and the following  general lemma. 

\begin{lemma}
	Suppose $E$ is a set with $N$ elements and $\pA \subseteq \Tup{}{E}$. If $u_{AB}=0$ in $\QG_{\pA}$ for all $A,B\in E^{N}$ such that $A$ has no repeated elements and $B$ does have repeated elements, then $\QG_{\pA}$ is commutative.  
\end{lemma}

\begin{proof}
	Let $0\leq k \leq N-1$. First, we claim that $u_{AB} = 0$ in $\QG_{\pA}$ for all $A,B \in E^{N-k}$ such that $A$ has no repeated elements and $B$ does have repeated elements. We proceed by induction on $k$, and the base case $k=0$ is a hypothesis of the lemma. For $k\geq 1$, suppose $A$ and $B$ are as above and have $N-k$ elements. If  $a\notin A$ then 
	\begin{equation*}
		u_{AB} = \sum_{b\in E}u_{AB} u_{ab}.
	\end{equation*}
	As each summand on the right lies in $I_{\pA}$ by the inductive hypothesis, we also have $u_{AB}=0$ in $ \QG_{\pA}$, as required. 
	
	Now, suppose $a,b,c,d\in E$ and consider $u_{ab}u_{cd}$. If $a=c$ or $b=d$, then $u_{ab}u_{cd}$ and $u_{cd}u_{ab}$  both equal 0 or 1, so assume $a\neq c$ and $b\neq d$. Then
	\begin{equation*}
		u_{ab}u_{cd} = \sum_{x\in E} u_{ab}u_{cd}u_{xb}.
	\end{equation*} 
	For $x\neq a,c$, the summand $u_{ab}u_{cd}u_{xb}$ lies in $I_{\pA}$ by the above claim, and  summand $u_{ab}u_{cd}u_{cb}$ equals 0 by the quantum symmetric group relations. Therefore
	\begin{equation*}
		u_{ab}u_{cd} = u_{ab}u_{cd}u_{ab}. 
	\end{equation*}
	As the term on the right is idempotent, we have $u_{ab}u_{cd} = u_{cd}u_{ab}$, and so $\QG_{\pA}$ is commutative, as required. 
	
\end{proof}

\section{Quantum symmetries for matroids by rank and girth}
\label{sec:QuantumSymmetriesByRank}

We begin this section by developing general descriptions of the quantum automorphism group $\QAut{\pB}{\sM}$ for matroids of rank $1$ and $2$. Given a matroid $\sM$ and $L\subseteq E(\sM)$, the \textit{deletion} of $L$ from $\sM$ is the matroid  $\sM\setminus L$ ground set $E(\sM) \setminus L$ and whose independent sets are $\pI(\sM\setminus L) = \{A\in \pI(\sM) \, : \, A\subseteq E(\sM) \setminus L\}$. 

\begin{proposition}
\label{prop:loopsFreeProduct}
Let $\sM$ be a matroid and denote by $L\subseteq E(\sM)$ the set of loops of $\sM$.  Then
    \begin{equation*}
		\QAut{\pI}{\sM}  \cong \freeprod{\QSn{L}}{\QAut{\pI}{\sM \setminus L}} {}
	\end{equation*}
	If $\sM$ has rank 1, then 
	\begin{equation*}
		\QAut{\pI}{\sM}  \cong \freeprod{\QSn{L}}{\QSn{E(\sM)\setminus L}}{}.
	\end{equation*}
\end{proposition}

\begin{proof}
Recall that $i\in E(\sM)$ is a loop of $\sM$ if $\{i\}$ is a dependent set. So  this proposition follows from the fact that $u_{ij} \in I(\QAut{\pI}{\sM})$ if $i\in L$  and $j\in E(\sM)\setminus L$, or $i\in E(\sM)\setminus L$  and $j\in L$. 
\end{proof}

Now suppose that $\sM$ is a loopless rank 2 matroid. The rank 1 flats form a partition of $E(\sM)$. Define a graph $\Gamma_2[\sM]$ in the following way. Let $V(\Gamma_2[\sM]) = E(\sM)$, and two vertices are connected by an edge if they lie in different rank-1 flats. 
\begin{proposition}
\label{prop:rank2BasesQAut}
	Let $\sM$ be a rank 2 matroid and let $L\subseteq E(\sM)$ be its set of loops. Then
	\begin{equation*}
		\QAut{\pB}{\sM} \cong \freeprod{\QSn{L}}{\;\QAut{}{\Gamma_2[\sM\setminus L] }}{}.
	\end{equation*}
\end{proposition}

\begin{proof}
	Given a graph $\Gamma$, by \cite[Lemma~5.7]{SpeicherWeber} its quantum automorphism group is the subgroup of $\QSn{V(\Gamma)}$ defined by the ideal
	\begin{equation*}
		\langle u_{ab}u_{cd} \, : \, (ac \in E(\Gamma) \text{ and } bd \notin E(\Gamma) )\text{ or } (ac \notin E(\Gamma) \text{ and } bd \in E(\Gamma) ) \rangle 
	\end{equation*}
	If $\sM$ is loopless, then the proposition follow from this description of $\QAut{}{\Gamma_2[\sM]}$ and the definition of $\QAut{\pB}{\sM}$. For the general case, apply Proposition \ref{prop:loopsFreeProduct}.
\end{proof}

\noindent Therefore, $\QAut{\pB}{\sM}$ are familiar quantum automorphism groups for matroids of rank 1 and 2. At the other end of the rank spectrum, we develop criterion which guarantees  that $\QAut{\pB}{\sM}$ is commutative for large classes of matroids of rank $\geq 4$. The \textit{girth} of the matroid $\sM$, denoted $\girth(\sM)$, is the size of the smallest circuit of $\sM$.  
\begin{theorem}
	\label{thm:girthgeq4}
	If $\sM$ is a matroid with $\girth(\sM) \geq 4$, then 
	\begin{equation*}
		\QAut{\pB}{\sM} = \QAut{\pI}{\sM} = C(\Aut(\sM)). 
	\end{equation*}
\end{theorem}
\noindent A rank $r$ matroid is \textit{paving} if its girth is $\geq r$.  Conjecturally, almost all matroids are paving \cite{CrapoRota, MayhewNewmanWelshWhittle}, and paving matroids of rank $\geq 4$ have commutative bases quantum automorphism group by Theorem \ref{thm:girthgeq4}. This leaves rank $r=3$, and there are several examples of matroids $\sM$ such that $\QAut{\pB}{\sM}$ is noncommutative, see \S\ref{sec:computations}.

\begin{proof}[Proof of Theorem \ref{thm:girthgeq4}]
    In view of Theorem \ref{thm:QAutIeqQAutB},  we may focus on proving the second equality in the theorem.  By Proposition \ref{prop:recoverOrdinaryAut}, it suffices to show that, for a matroid $\sM$ with $\girth(\sM) \geq 4$, we have
\[
    u_{ab}u_{cd}u_{ab} = u_{ab}u_{cd} \; \text{ in } \;  \QAut{\pI}{\sM}
\]
for all $a,b,c,d \in E(\sM)$. Indeed, this formula implies that $u_{ab}$ and $u_{cd}$ commute since $u_{ab}u_{cd}u_{ab}$ is self-adjoint.  

\noindent \textbf{Case 1}.
If  $a = c$ and $b=d$, then
\[
    u_{ab}u_{cd}u_{ab} = u_{ab}^2u_{ab} = u_{ab}u_{cd}.
\]
\noindent \textbf{Case 2}. If $a \neq c$ and $b=d$   then
\[
    u_{ab}u_{cd}u_{ab} = u_{ab}u_{cb}u_{ab} = 0 = u_{ab}u_{cb}.
\]
\noindent \textbf{Case 3}. If $a = c$ and $b\neq d$, argue as in Case 2. 

\medskip 

\noindent \textbf{Case 4}. Suppose $a \neq c$ and $b \neq d$. Then
\[
    u_{ab}u_{cd} = u_{ab}u_{cd}\sum_{x \in E }u_{ax}.
\]
We must show that $u_{ab}u_{cd}u_{ax}=0$ in $\QAut{\pI}{\sM}$ for every $x \in E(\sM)\setminus \{b\}$. Since $(a,c,a)$ has  repeated elements, it is not an independent tuple. On the other hand, $(b,d,x)$ is an independent tuple for $x\neq b,d$ by the hypothesis $\girth(\sM) \geq 4$. If $x = d$, then
\[
    u_{ab}u_{cd}u_{ax} = u_{ab}u_{cd}u_{ad} = 0
\]
since $a \neq c$. In either case, we have 
\[
    u_{ab}u_{cd}u_{ax} = 0 \;\text{ in }\; \QAut{\pI}{\sM}  \; \text{ for } \; x\in E(\sM)\setminus \{b\}
\] 
as required.
\end{proof}

\section{Lov\'{a}sz's theorem for matroids}
\label{sec:Lovasz}

So far our main purpose was to introduce and study interesting new classes of noncommutative algebras, namely the various quantum automorphism groups of matroids.
But it is natural to ask if we can also learn something about matroids from their quantum automorphisms.
While we currently lack the technology to obtain a deep result in that direction. Here we want to sketch an idea that we find intriguing.

Our point of departure is a celebrated result in graph theory.
Lov\'asz in \cite{Lovasz} derives the following characterization of isomorphic graphs in terms of graph homomorphism counts. 
\begin{theorem}
	\label{thm:Lovasz}
	Two graphs $\Gamma_1$ and $\Gamma_2$, without multiple edges, are isomorphic if and only if
	\begin{equation*}
	|\Hom(\Gamma, \Gamma_1)| = |\Hom(\Gamma,\Gamma_2)|
	\end{equation*}
	for all graphs $\Gamma$. 
\end{theorem}
\noindent Exciting recent work of Mancinska and Roberson \cite{MancinskaRoberson} characterizes \textit{quantum} isomorphic graphs in terms of homomorphism counts from \textit{planar} graphs.
\begin{theorem}
	\label{thm:quantumLovasz}
	Two graphs $\Gamma_1$ and $\Gamma_2$, without multiple edges, are quantum isomorphic if and only if
	\begin{equation*}
	|\Hom(\Gamma, \Gamma_1)| = |\Hom(\Gamma,\Gamma_2)|
	\end{equation*}
	for all planar graphs $\Gamma$. 
\end{theorem}
\noindent
In this section, we derive a matroidal analog of Theorem \ref{thm:Lovasz}, which is the first critical step to see if \textit{quantum isomorphic} matroids admit a characterization in terms of counts of maps of matroids.

Matroid theory is somewhat complicated from a categorical point of view.
Namely, there are at least two competing concepts for morphisms, which are discussed and compared in Chapters 8 and 9 of \cite{White}.
One type of such morphisms are the strong maps, which we will define next.
Let $\sM_1$ and $\sM_2$. A \textit{strong map} $\varphi:\sM_1\to \sM_2$ is a function
\begin{equation*}
	\varphi:E(\sM_1) \sqcup \{e\} \to E(\sM_2) \sqcup \{e\}
\end{equation*}
such that $\varphi(e) = e$ and the inverse image of a flat of $\sM_2\oplus \sU(0,e)$ is a flat of $\sM_1\oplus \sU(0,e)$. The \textit{image} of $\varphi$, denoted by $\varphi(\sM_1)$, is the restriction $\sM_2|_{E_1'}$ where $E_1' = \varphi(E(\sM_1) \cup \{e\}) \cap E(\sM_2)$. (Given $L\subseteq E(\sM)$, the restriction $\sM_{|L}$ is simply the deletion of $E(\sM) \setminus L$ from $\sM$ as defined at the beginning of \S\ref{sec:QuantumSymmetriesByRank}.)  Denote by $\Hom(\sM_1,\sM_2)$ the set of strong maps $\varphi: \sM_1 \to \sM_2$. An \textit{embedding} of matroids is a strong map $\varphi:\sM_1\hookrightarrow \sM_2$ such that $\varphi:E(\sM_1)\sqcup \{e\} \to E(\sM_2)\sqcup \{e\}$ is injective and $\sM_1 \cong \varphi(\sM_1)$. 
We are now equipped for the main result of this section.

\begin{theorem}
\label{thm:matroidLovasz_v2}
	Two matroids $\sM_1$ and $\sM_2$ are isomorphic to each other if and only if 
	\begin{equation*}
		|\Hom(\sM_1, \sL)| = |\Hom(\sM_2, \sL)|
	\end{equation*}
	for all matroids $\sL$. 
\end{theorem}

\noindent Note the crucial difference between Theorems \ref{thm:Lovasz} and \ref{thm:matroidLovasz_v2} is that Theorem \ref{thm:Lovasz} concerns counts of graph homomorphisms \textit{to} the graphs $\Gamma_1$ and $\Gamma_2$, whereas Theorem \ref{thm:matroidLovasz_v2} concerns counts of strong maps \textit{from} the matroids $\sM_1$ and $\sM_2$. This is reflected in the fact that if $\varphi:\Gamma_1 \to \Gamma_2$ is a graph homomorphism between graphs with the same number of vertices such that $\varphi(\Gamma_1) = \Gamma_2$, then $\varphi$ is an isomorphism.   However, if $\psi:\sM_1\to \sM_2$ is an embedding of matroids on ground sets of the same size, then $\psi$ is an isomorphism.

\begin{lemma}
	\label{lem:strongMapInjectiveIso_v2}
The matroids $\sM_1$ and $\sM_2$ are isomorphic if and only if there are surjective strong maps $\varphi_1:\sM_1 \to \sM_2$ and $\varphi_2:\sM_2 \to \sM_1$. 
\end{lemma}

\begin{proof}
	Since the maps $\varphi_1:E(\sM_1)\sqcup \{e\} \to E(\sM_2) \sqcup \{e\}$ and  $\varphi_2:E(\sM_2)\sqcup \{e\} \to E(\sM_1) \sqcup \{e\}$ are both surjective maps of finite sets, they are bijective. Since $\varphi_1$ and $\varphi_2$ are strong maps of the same size, the matroids $\sM_1$ and $\sM_2$ have the same rank, and hence $\varphi_1$ and $\varphi_2$ define isomorphisms, see \cite[Proposition~8.1.6]{Kung}. 
\end{proof}

\noindent
Denote by $\Emb(\sM_1,\sM_2)$ the set of embeddings $\sM_1 \hookrightarrow \sM_2$ and $\Surj(\sM_1,\sM_2)$ the set of surjective strong maps $\sM_1\to \sM_2$, respectively. Given another matroid $\sN$,  set 
\begin{equation*}
  \Hom(\sM_1,\sM_2; \sN) = \{\varphi \in \Hom(\sM_1,\sM_2) \, : \, \varphi(\sM_1) \cong \sN \}.
\end{equation*}
The following lemma and its proof are  matroidal analogs of \cite[Equation 6]{Lovasz}. 
\begin{lemma}
  \label{lem:decomposeHom}
  Given matroids $\sM_1$ and $\sM_2$, we have 	
  \begin{equation*}
    |\Hom(\sM_1,\sM_2)|= \sum_{\sN\in \pM} |\Hom(\sM_1,\sM_2;\sN)| = \sum_{\sN\in \pM} \frac{|\Surj(\sM_1,\sN)||\Emb(\sN,\sM_2)|}{|\Aut(\sN)|}
  \end{equation*}
  where $\pM$ is a set of representatives of the isomorphism classes of matroids. 
\end{lemma}

\begin{proof}
  Given  $\varphi \in \Hom(\sM_1,\sM_2; \sN)$, define
  \begin{equation*}
    E_{\varphi} = \{(\pi, \psi) \, : \, \pi \in \Surj(\sM_1,\sN), \ \psi \in \Emb(\sN,\sM_2), \ \varphi =  \psi \circ \pi  \}
  \end{equation*}
  We claim the following:

  \medskip
  
  \begin{enumerate}
  \item $|E_{\varphi}| = |\Aut(\sN)|$;
  \item $E_{\varphi_1} \cap E_{\varphi_2} = \emptyset$ for different $\varphi_1,\varphi_2 \in  \Hom(\sM_1,\sM_2;\sN)$;
  \item every pair $(\pi, \psi)$ with $ \pi \in \Surj(\sM_1,\sN) $ and $\psi \in \Emb(\sN,\sM_2)$ is in some $E_{\varphi}$.
  \end{enumerate}
  
  \medskip

  \noindent Properties (2) and (3) are clear, so consider (1). First, we claim that $E_{\varphi} \neq \emptyset$. By \cite[Lemma~8.1.4]{Kung}, we can factor $\varphi$ into strong maps $\pi':\sM_1 \to \varphi(\sM_1)$ and $\psi':\varphi(\sM_1) \to \sM_2$. Fix an isomorphism $\sigma: \varphi(\sM_1) \to \sN$. Then $(\sigma \circ \pi', \psi'\circ \sigma^{-1}) \in \Hom(\sM_1, \sM_2; \sN)$.  
  
  Fix a pair $(\pi_0, \psi_0) \in E_{\varphi}$ and set
  \begin{equation*}
    E'_{\varphi} = \{ (\sigma \circ \pi_0, \psi_0 \circ \sigma^{-1}) \, : \, \sigma \in \Aut(\sN)\}
  \end{equation*}
  If $\sigma_1, \sigma_2 \in \Aut(\sN)$ are distinct, then $\sigma_1 \circ \pi_0 \neq \sigma_2 \circ \pi_0$ since $\pi_0$ is surjective. So $|E_{\varphi}'| = |\Aut(\sN)|$ and $E_{\varphi}' \subseteq E_{\varphi}$. Conversely, suppose $(\pi,\psi) \in E_{\varphi}$. The images $\psi(\sN)$ and $\psi_{0}(\sN)$ both equal $\varphi(\sM_1)$, and $\psi,\psi_{0}:\sN \to \varphi(\sM_1)$ are isomorphisms. Then $\sigma = \psi^{-1} \circ \psi_{0}$ is an automorphism of $\sN$ and $(\pi,\psi) = (\sigma \circ \pi_{0}, \psi_{0} \circ \sigma^{-1})$. Therefore  $E_{\varphi} = E_{\varphi}'$, which proves (1). 
  
  By these properties, $\{E_{\varphi} \, : \, \varphi \in \Hom(\sM_1,\sM_2;\sN)\}$ form a partition of a set of size $|\Surj(\sM_1,\sN)||\Emb(\sN,\sM_2)|$ into $|\Hom(\sM_1,\sM_2;\sN)|$ subsets of size $|\Aut(\sN)|$. So
  \begin{equation*}
    |\Hom(\sM_1,\sM_2;\sN)||\Aut(\sN)| =  |\Surj(\sM_1,\sN)||\Emb(\sN,\sM_2)|,
  \end{equation*}
  from which the theorem follows. 
\end{proof}

\begin{proof}[Proof of Theorem~\ref{thm:matroidLovasz_v2}]
  We prove that $|\Surj(\sM_1, \sL)| = |\Surj(\sM_2, \sL)|$ for all matroids  $\sL$. If  $\sL=\sM_2$ then $\Surj(\sM_1,\sM_2)\neq \emptyset$, and if $\sL=\sM_1$ then $\Surj(\sM_2,\sM_1)\neq \emptyset$. The theorem then follows from Lemma \ref{lem:strongMapInjectiveIso_v2}. 
  
  We proceed by induction on $|E(\sL)|$. 	The equality  $|\Surj(\sM_1, \sL)| = |\Surj(\sM_2, \sL)|$ is clear when $|E(\sL)| = 1$. Now, observe that 
  \begin{equation*}
    |\Surj(\sM_i,\sL)| = \sum_{\substack{\sN\in \pM \\ |E(\sN)|=|E(\sL)|}} |\Hom(\sM_i, \sL; \sN)|
  \end{equation*}
  By Lemma \ref{lem:decomposeHom}, we have 
  \begin{equation*}
    |\Hom(\sM_i,\sL)| = |\Surj(\sM_i,\sL)| +  \sum_{\substack{\sN\in \pM \\ |E(\sN)|<|E(\sL)|}} \frac{|\Surj(\sM_i,\sN)||\Emb(\sN,\sL)|}{|\Aut(\sN)|}.
  \end{equation*}
  Since $|\Hom(\sM_1,\sL)| = |\Hom(\sM_2,\sL)|$, the difference $|\Surj(\sM_1,\sL)| - |\Surj(\sM_2,\sL)|$ is
  \begin{equation*}
    \sum_{\substack{\sN\in \pM \\ |E(\sN)|<|E(\sL)|}} \bigl(|\Surj(\sM_2,\sN)| - |\Surj(\sM_1,\sN)|\bigr) \frac{|\Emb(\sN,\sL)|}{|\Aut(\sN)|}.
  \end{equation*}
  This equals 0 by the inductive hypothesis. 
\end{proof}

The key question, however, remains open for now:
\begin{question}
  Is there a matroidal analog of Theorem~\ref{thm:quantumLovasz}?
\end{question}

\section{Computational results}
\label{sec:computations}

The study of the quantum automorphism groups of a given matroid, beyond the simplest cases, is impractical to do by hand.
Therefore, we rely on computer algebra to study specific examples.
Our primary goal is to determine, for as many matroids $\sM$ as possible, whether $\QAut{\pB}{\sM}$ and $\QAut{\pC}{\sM}$ are commutative, which in view of Proposition \ref{prop:recoverOrdinaryAut} is equivalent to being isomorphic to $C(\Aut(\sM))$.
In view of Remark~\ref{rem:coefficients} it suffices to compute with quantum permutation groups with rational coefficients.
We employ the free open-source software \Oscar~\cite{OSCAR-book, Oscar} which features two different methods for computing noncommutative Gr\"{o}bner bases.
The results are given as tables in Appendix~\ref{sec:tables}.

\subsection{Algorithms and their implementations}
Hilbert's basis theorem says that every ideal $I$ in a commutative polynomial ring is finitely generated, i.e., such a ring is \emph{Noetherian}.
Furthermore, the ideal $I$ has particularly useful generating systems, known as \emph{Gr\"{o}bner bases}, which, e.g., allow to decide whether a given polynomial is contained in $I$ or not.
Buchberger's algorithm is a classical method to convert any finite generating system of $I$ into a Gr\"{o}bner basis.
For details, we recommend the textbook by von zur Gathen and Gerhard \cite[\S21]{vonzurGathenGerhard:2003}.

Our computations employ a noncommutative generalization of Buchberger's algorithm, which are substantially more subtle than their commutative analogs.
We consider the polynomial ring $R=\Q\langle X\rangle$ with a finite set $X$ of noncommuting variables and rational coefficients.
The biggest obstacle to overcome for \emph{any} algorithm dealing with ideals of $R$ is the fact that $R$ is not Noetherian. 
Nonetheless, here is an outline of an algorithm by Xiu~\cite[\S4]{xiu_2012}; precise definitions of all italicized terms may be found in this text.
That algorithm and its \Oscar implementation were already employed in \cite{LevandovskyyEderSteenpassSchmidtSchanzWeber}.

\begin{algorithm}[Noncommutative Buchberger]
  \label{algo:buchberger}
  Let $I\subset R$ be a (two-sided) ideal of $R$ generated by a finite list $\pG$.
  \begin{enumerate}
  \item For each pair $f$ and $g$ in $\pG$, including the case $f = g$, we compute the set of 
     all (nontrivial) \textit{obstructions} of $f$ and  $g$ and call it $o_{f,g}$ (an obstruction is a noncommutative analog of a syzygy). The set $B$ is initialized as the union of all obstruction sets $o_{f, g}$ of $\pG$.
  \item If  $B = \varnothing$, return $\pG$.
    Otherwise select an obstruction $o_{f, g}(w_f, w_f', w_g, w_g') \in B$ using a fair strategy and delete it from $B$.
    The four monomials $w_f$, $w_f'$, $w_g$, $w_g'$ that specify the obstruction satisfy 
    $w_f \operatorname{LT}_{\sigma}(f) w_f' = w_g \operatorname{LT}_{\sigma}(g) w_g'$. Here,  $\operatorname{LT}_{\sigma}(x)$ denotes the \textit{leading term}
    of the polynomial $x$ with respect to the degree reverse lexicographic ordering $\sigma$.
    A selection strategy is \textit{fair} if it ensures that every obstruction is eventually selected.
  \item Compute the  $S$-polynomial  $S = S_{f, g}(w_{f}, w_{f}'; w_{g}, w_{g}')$ and its \textit{normal remainder} $S' \coloneqq \operatorname{NR}_{\sigma, \pG}(S)$.
    If $S' = 0$, continue with step (2).
  \item Append  $S'$ to $\pG$. Append the obstructions  $o_{h,S'}$ for all $h\in\pG$ to $B$ (including $h = S'$). 
    Continue with step (2).
  \end{enumerate}
\end{algorithm}

Since $R$ is not Noetherian, that algorithm may not terminate, but if it does, then the output is a \emph{Gr\"obner basis} of $I$, with respect to the degee reverse lexicographic ordering.
Term orderings are subtle in the noncommutative case, whence we stick to this particular term ordering throughout.
In addition to this basic algorithm, we use some simple optimizations to make the computation run faster.
The first step for this is to make the set of obstructions $B$ into a priority queue that is sorted in the degree reverse lexicographic order, which is an improvement discussed in the PhD thesis of Keller~\cite{Keller-thesis}.
Moreover, for computing the obstructions we use the Aho--Corasick algorithm \cite{Aho+Corasick:1975} for efficiently finding substrings in a given text to improve the function that checks for divisibility by the partial Gr\"{o}bner basis.
This works by maintaining an Aho--Corasick automaton during the computation that contains all elements of the Gr\"{o}bner basis computed so far. When checking for a given monomial whether it is divisible by an element of the Gr\"{o}bner basis, one looks for the first element in the automaton that matches a substring of the monomial.
If such an element exists, that monomial is divisible by the partial Gr\"{o}bner basis.
The Algorithm~\ref{algo:buchberger} was implemented in \Oscar directly.

La Scala and Levandovskyy proposed a different way of dealing with not finitely generated ideals \cite{LaScalaLevandovskyy:2009}.
To the generating system $\pG$ of an ideal $I$ in the noncommutative polynomial ring $R$ they associate an ideal, called the \emph{letterplace ideal} of $I$, which lives in a commutative polynomial ring, but with infinitely many variables.
By restricting to subrings with finitely many variables, say $d$, this method allows to employ standard commutative Gr\"{o}bner bases to obtain \emph{truncated Gr\"{o}bner bases} of $I$.
The number $d$ is called the \emph{degree bound} of the truncation, and it is part of the input.

In contrast with the Buchberger algorithm sketched above, the advantage of the Letterplace algorithm is its termination for every input.
The drawback is that it does not directly yield a Gr\"{o}bner basis even if it exists.
On the bright side, we have the following key result.
\begin{theorem}[{La Scala and Levandovsky \cite[Cor.~3.19]{LaScalaLevandovskyy:2009}}]
  Let $I \subset R$ be a graded two-sided ideal with a finite homogeneous basis whose polynomials are all of degree $d$.
  Denote by $\pG_{d-1}$ the truncated Gr\"{o}bner basis of $I$ up to degree $d$.
  If $\pG_{d-1} = \pG_{2d-2}$, then $\pG_{d-1}$ is a Gr\"{o}bner basis of $I$.
\end{theorem}
That is to say, if the sequence of truncated Gr\"{o}bner bases stabilizes enough, then the final truncated Gr\"{o}bner basis is actually a proper Gr\"{o}bner basis.
This method is implemented in \Letterplace \cite{letterplace}, which is also available in \Oscar.

Either method for computing noncommutative Gr\"{o}bner bases leads to a procedure for semi-deciding the commutativity of the quantum automorphism group $\QAut{\pB}{\sM}$. Given a Gr\"{o}bner basis of the ideal $I(\QAut{\pB}{\sM})$ we can check if all commutators $u_{ij}u_{k\ell}-u_{k\ell}u_{ij}$ lie in $I(\QAut{\pB}{\sM})$. Similarly for $\QAut{\pC}{\sM}$. If the Gr\"{o}bner basis computation does not terminate, we cannot say anything. Nevertheless, using these techniques we are able to determine if $\QAut{\pB}{\sM}$ or $\QAut{\pC}{\sM}$ are commutative for numerous matroids of small rank and ground set, and so we pose the following question.

\begin{question}
  Is commutativity of $\QAut{\pB}{\sM}$ or $\QAut{\pC}{\sM}$ decidable?
\end{question}

\subsection{Data}
Tables \ref{Tab:computational-results-1}, \ref{Tab:computational-results-2}, \ref{Tab:computational-results-3}, and \ref{Tab:computational-results-4} group isomorphism classes of matroids $\sM$ that have small rank and ground set based on whether $\QAut{\pB}{\sM}$ or $\QAut{\pC}{\sM}$ are commutative or noncommutative. For numerous matroids, we were only able to determine commutativity of $\QAut{\pB}{\sM}$, and this data is recorded in tables \ref{Tab:computational-results-5} and \ref{Tab:computational-results-6}.  We use the matroids from \polymake's database \texttt{polyDB} \cite{DMV:polymake,polyDB}, which is based on the data obtained in \cite{MatsumotoMoriyamaImaiBremner}.  In this database, matroids are recorded via a \textit{reverse lexicographic basis encoding}. This encoding works in the following way.  A matroid $\sM$ with  rank $r$ and ground set $\{1, \ldots, n\}$ is encoded as a binary string of length $\binom{n}{r}$. The positions of the characters in this string correspond to the $r$-element subsets of $\{1, \ldots, n\}$  in reverse lexicographic order. A \texttt{0} means that the corresponding subset is not a basis of $\sM$, and a \texttt{1} means that the corresponding subset is a basis (recorded in this database as \texttt{*}). For brevity, we record these strings in hexadecimal and pad the beginning with \texttt{0}'s so that the string has $\lceil\frac{1}{4} \binom{n}{r} \rceil$ characters. 

\begin{example}
	Consider the Fano matroid $\sF$ from Example \ref{ex:fano}. The bases of this matroid are all $3$-element subsets of $\{1,\ldots,7\}$ except for 
	\begin{equation*}
		123, 145, 246, 356, 347, 257, 167.
	\end{equation*}
	Here, $ijk$ is short for the subset $\{i,j,k\}$. There are 35 three-element subsets of $\{1,\ldots,7\}$, and in reverse lexicographic these are
	\begin{gather*}
		123, 124, 134, 234, 125, 135, 235, 145, 245, 345, 126, 136, 236, 146, 246, 346, 156, 256, 356, 456, \\
		127, 137, 237, 147, 247, 347, 157, 257, 357, 457, 167, 267, 367, 467, 567.
	\end{gather*}
	The revlex basis encoding of $\sF$ reads
	\begin{equation*}
		\texttt{011\;1111\;0111\;1110\;1110\;1111\;1101\;0110\;1111}
	\end{equation*}
	and so the hexadecimal encoding is \texttt{3f7eefd6f}. 
\end{example}

Our tables in Appendix~\ref{sec:tables} comprise the complete data for all isomorphism classes of matroids with ground set sizes $n\leq 6$ and rank $r \leq 2$.
We have partial data for $(r,n) = (3,6), (3,7)$.
The columns in each table record the following information.
\begin{enumerate}
\item Matroid: the matroid $\sM$ in hexadecimal encoding;
\item $|E(\sM)|$: the size of the ground set of $\sM$;
\item $\rank(\sM)$: the rank of $\sM$;
\item $\girth(\sM)$: the girth of the matroid;
\item \# nonbases: the number $\binom{n}{r} - |\pB(\sM)|$ where $n=|E(\sM)|$ and $r = \rank(\sM)$;
\item $|\Aut(\sM)|$: the order of the classical automorphism group of $\sM$;
\item $d_{\pB}(\sM)$: degree of the Gr\"{o}bner basis computed for $I(\QAut{\pB}{\sM})$.
\end{enumerate}
All our results have been obtained with the \Oscar implementation of the noncommutative Buchberger Algorithm~\ref{algo:buchberger}.

Since we are dealing with a computational task which a priori is not finite, it is useful to start with matroids for which it seems likely that the computation can be finished.
Our experiments suggest the following heuristics.
We start as low as $n=2$, and then we stepwise increase the number of elements in the ground set.
For each fixed $n$ we sort the matroids by increasing rank $r$.
The set of $(r,n)$--matroids is further sorted by taking those first for which the parameter $|\pB| \cdot \left( \sum_{k=1}^r{ n \choose k } - |\pB| \right)$ is small; the latter number is the number of relations defining the ideal $I(\QAut{\pB}{\sM})$; cf.\ \eqref{eq:IA}.
We stopped the calculation as soon as the elapsed time exceeded one week per matroid.
This resulted in a total of $207$ matroids for which we could compute the commutativity of $\QAut{\pB}{\sM}$, out of which we could determine $\QAut{\pC}{\sM}$ for $85$ matroids.

All computations were done on the HPC-Cluster at Technische Universit\"{a}t Berlin.
Specifically, we used an openSUSE Leap 15.4 distribution on a cluster, where each node has 2x AMD EPYC 7302 16-core processors with 1024 gigabytes of RAM.
The computation times vary a lot, matroids with $n\geq6$ elements routinely take $24$ hours or more. The total computation time for the base versions was about 33 days.

\bibliographystyle{abbrv}
\bibliography{bibliographie}
\label{sec:biblio}

\appendix
 
\section{Tables}
\label{sec:tables}

We distinguish whether the quantum automorphism groups $\QAut{\pB}{\sM}$ and $\QAut{\pC}{\sM}$ of the matroid $\sM$ are commutative or not.
Combining the results from Tables \ref{Tab:computational-results-3} and \ref{Tab:computational-results-4} reveals, e.g., that there is no direct generalization of Theorem~\ref{thm:QAutCInQAutB} to arbitrary matroids.
Several uniform matroids can be spotted easily, as they are the only ones without nonbases.
For instance, $\QAut{\pB}{\MatU{2}{4}}$ is noncommutative but $\QAut{\pC}{\MatU{2}{4}}$ is commutative (\texttt{3f} in Table~\ref{Tab:computational-results-4}).
Our computations for the Fano matroid $\sF$ from Example~\ref{ex:fano} did not terminate; so it does not occur here.

\begin{center}
\small
\begin{longtable}{lrrrrrr}
  \caption{$\QAut{\pB}{\sM}$ and $\QAut{\pC}{\sM}$ both noncommutative} \label{Tab:computational-results-1} \\
  \toprule
  Matroid & $|E(\sM)|$ & $\rank(\sM)$ & $\girth(\sM)$ & \#nonbases & $|\Aut(\sM)|$ & $d_{\pB}(\sM)$ \\ \midrule
  \endfirsthead
  \caption{$\QAut{\pB}{\sM}$ and $\QAut{\pC}{\sM}$ both noncommutative (continued)} \\  
  Matroid & $|E(\sM)|$ & $\rank(\sM)$ & $\girth(\sM)$ & \#nonbases & $|\Aut(\sM)|$ & $d_{\pB}(\sM)$ \\ \midrule
  \endhead
  \endfoot
  \endlastfoot
  \texttt{f} & 4 & 1 & 2 & 0 & 24 & 3 \\
  \texttt{3} & 4 & 1 & 1 & 2 & 4 & 2 \\
  \texttt{1e} & 4 & 2 & 2 & 2 & 8 & 2 \\
  \texttt{01} & 4 & 2 & 1 & 5 & 4 & 2 \\
  \texttt{1f} & 5 & 1 & 2 & 0 & 120 & 3 \\
  \texttt{0f} & 5 & 1 & 1 & 1 & 24 & 3 \\
  \texttt{07} & 5 & 1 & 1 & 2 & 12 & 2 \\
  \texttt{03} & 5 & 1 & 1 & 3 & 12 & 2 \\
  \texttt{01} & 5 & 1 & 1 & 4 & 24 & 3 \\
  \texttt{1ef} & 5 & 2 & 2 & 2 & 8 & 2 \\
  \texttt{07e} & 5 & 2 & 2 & 4 & 12 & 2 \\
  \texttt{036} & 5 & 2 & 1 & 6 & 8 & 2 \\
  \texttt{013} & 5 & 2 & 1 & 7 & 12 & 2 \\
  \texttt{00f} & 5 & 2 & 2 & 6 & 24 & 3 \\
  \texttt{003} & 5 & 2 & 1 & 8 & 4 & 2 \\
  \texttt{001} & 5 & 2 & 1 & 9 & 12 & 2 \\
  \texttt{3f} & 6 & 1 & 2 & 0 & 720 & 3 \\
  \texttt{1f} & 6 & 1 & 1 & 1 & 120 & 3 \\
  \texttt{0f} & 6 & 1 & 1 & 2 & 48 & 3 \\
  \texttt{07} & 6 & 1 & 1 & 3 & 36 & 2 \\
  \texttt{03} & 6 & 1 & 1 & 4 & 48 & 3 \\
  \texttt{01} & 6 & 1 & 1 & 5 & 120 & 3 \\
  \texttt{3dff} & 6 & 2 & 2 & 2 & 16 & 2 \\
  \texttt{3dfe} & 6 & 2 & 2 & 3 & 48 & 3 \\
  \texttt{0fdf} & 6 & 2 & 2 & 4 & 12 & 2 \\
  \texttt{0fdc} & 6 & 2 & 2 & 6 & 72 & 2 \\
  \texttt{06cf} & 6 & 2 & 1 & 7 & 8 & 2 \\
  \texttt{0267} & 6 & 2 & 1 & 9 & 48 & 3 \\
  \texttt{01ff} & 6 & 2 & 2 & 6 & 48 & 3 \\
  \texttt{01fe} & 6 & 2 & 2 & 7 & 48 & 3 \\
  \texttt{00ee} & 6 & 2 & 1 & 9 & 12 & 2 \\
  \texttt{0067} & 6 & 2 & 1 & 10 & 8 & 2 \\
  \texttt{0066} & 6 & 2 & 1 & 11 & 16 & 2 \\
  \texttt{0023} & 6 & 2 & 1 & 12 & 36 & 2 \\
  \texttt{001f} & 6 & 2 & 2 & 10 & 120 & 3 \\
  \texttt{000f} & 6 & 2 & 1 & 11 & 24 & 3 \\
  \texttt{0007} & 6 & 2 & 1 & 12 & 12 & 2 \\
  \texttt{0003} & 6 & 2 & 1 & 13 & 12 & 2 \\
  \texttt{0001} & 6 & 2 & 1 & 14 & 48 & 3 \\
  \texttt{0fff0} & 6 & 3 & 2 & 8 & 48 & 3 \\
  \texttt{079e3} & 6 & 3 & 2 & 10 & 8 & 3 \\
  \texttt{079e0} & 6 & 3 & 2 & 12 & 48 & 3 \\
  \texttt{00413} & 6 & 3 & 1 & 16 & 48 & 3 \\
  \texttt{001ef} & 6 & 3 & 2 & 12 & 8 & 3 \\
  \texttt{0007e} & 6 & 3 & 2 & 14 & 12 & 3 \\
  \texttt{00036} & 6 & 3 & 1 & 16 & 8 & 3 \\
  \texttt{00013} & 6 & 3 & 1 & 17 & 12 & 3 \\
  \texttt{0000f} & 6 & 3 & 2 & 16 & 48 & 3 \\
  \texttt{00003} & 6 & 3 & 1 & 18 & 8 & 3 \\
  \texttt{00001} & 6 & 3 & 1 & 19 & 36 & 3 \\
  \bottomrule
\end{longtable}

\end{center}

\begin{center}
\small
\begin{longtable}{lrrrrrr}
  \caption{$\QAut{\pB}{\sM}$ and $\QAut{\pC}{\sM}$ both commutative} \label{Tab:computational-results-2} \\
  \toprule
  Matroid & $|E(\sM)|$ & $\rank(\sM)$ & $\girth(\sM)$ & \#nonbases & $|\Aut(\sM)|$ & $d_{\pB}(\sM)$ \\ \midrule
  \endfirsthead
  \caption{$\QAut{\pB}{\sM}$ and $\QAut{\pC}{\sM}$ both commutative (continued)} \\  
  Matroid & $|E(\sM)|$ & $\rank(\sM)$ & $\girth(\sM)$ & \#nonbases & $|\Aut(\sM)|$ & $d_{\pB}(\sM)$ \\ \midrule
  \endhead
  \endfoot
  \endlastfoot
  \texttt{3} & 2 & 1 & 2 & 0 & 2 & 2 \\
  \texttt{1} & 2 & 1 & 1 & 1 & 1 & 2 \\
  \texttt{7} & 3 & 1 & 2 & 0 & 6 & 2 \\
  \texttt{3} & 3 & 1 & 1 & 1 & 2 & 2 \\
  \texttt{1} & 3 & 1 & 1 & 2 & 2 & 2 \\
  \texttt{7} & 4 & 1 & 1 & 1 & 6 & 2 \\
  \texttt{1} & 4 & 1 & 1 & 3 & 6 & 2 \\
  \texttt{0b} & 4 & 2 & 1 & 3 & 6 & 2 \\
  \texttt{07} & 4 & 2 & 2 & 3 & 6 & 2 \\
  \texttt{03} & 4 & 2 & 1 & 4 & 2 & 2 \\
  \texttt{007} & 5 & 2 & 1 & 7 & 6 & 2 \\
  \texttt{02cb3} & 6 & 3 & 1 & 12 & 8 & 3 \\
  \bottomrule
\end{longtable}

\end{center}

\clearpage

\begin{center}
\small
\begin{longtable}{lrrrrrr}
  \caption{$\QAut{\pB}{\sM}$ commutative but $\QAut{\pC}{\sM}$ noncommutative} \label{Tab:computational-results-3} \\
  \toprule
  Matroid & $|E(\sM)|$ & $\rank(\sM)$ & $\girth(\sM)$ & \#nonbases & $|\Aut(\sM)|$ & $d_{\pB}(\sM)$ \\ \midrule
  \endfirsthead
  \caption{$\QAut{\pB}{\sM}$ commutative but $\QAut{\pC}{\sM}$ noncommutative (continued)} \\  
  Matroid & $|E(\sM)|$ & $\rank(\sM)$ & $\girth(\sM)$ & \#nonbases & $|\Aut(\sM)|$ & $d_{\pB}(\sM)$ \\ \midrule
  \endhead
  \endfoot
  \endlastfoot
  \texttt{01c7e} & 6 & 3 & 2 & 11 & 36 & 3 \\
  \texttt{00c36} & 6 & 3 & 1 & 14 & 12 & 3 \\
  \texttt{00007} & 6 & 3 & 1 & 17 & 12 & 3 \\
  \bottomrule
\end{longtable}

\end{center}

\begin{center}
\small
\begin{longtable}{lrrrrrr}
  \caption{$\QAut{\pB}{\sM}$ noncommutative but  $\QAut{\pC}{\sM}$ commutative} \label{Tab:computational-results-4} \\
  \toprule
  Matroid & $|E(\sM)|$ & $\rank(\sM)$ & $\girth(\sM)$ & \#nonbases & $|\Aut(\sM)|$ & $d_{\pB}(\sM)$ \\ \midrule
  \endfirsthead
  \caption{$\QAut{\pB}{\sM}$ noncommutative but  $\QAut{\pC}{\sM}$ commutative (continued)} \\  
  Matroid & $|E(\sM)|$ & $\rank(\sM)$ & $\girth(\sM)$ & \#nonbases & $|\Aut(\sM)|$ & $d_{\pB}(\sM)$ \\ \midrule
  \endhead
  \endfoot
  \endlastfoot
  \texttt{3f} & 4 & 2 & 3 & 0 & 24 & 3 \\
  \texttt{1f} & 4 & 2 & 2 & 1 & 4 & 2 \\
  \texttt{3ff} & 5 & 2 & 3 & 0 & 120 & 3 \\
  \texttt{1ff} & 5 & 2 & 2 & 1 & 12 & 2 \\
  \texttt{0b7} & 5 & 2 & 1 & 4 & 24 & 3 \\
  \texttt{07f} & 5 & 2 & 2 & 3 & 12 & 2 \\
  \texttt{037} & 5 & 2 & 1 & 5 & 4 & 2 \\
  \texttt{7fff} & 6 & 2 & 3 & 0 & 720 & 3 \\
  \texttt{3fff} & 6 & 2 & 2 & 1 & 48 & 3 \\
  \texttt{16ef} & 6 & 2 & 1 & 5 & 120 & 3 \\
  \texttt{0fff} & 6 & 2 & 2 & 3 & 36 & 2 \\
  \texttt{06ef} & 6 & 2 & 1 & 6 & 12 & 2 \\
  \texttt{00ef} & 6 & 2 & 1 & 8 & 12 & 2 \\
  \texttt{07df3} & 6 & 3 & 2 & 8 & 16 & 3 \\
  \texttt{003ff} & 6 & 3 & 3 & 10 & 120 & 3 \\
  \texttt{001ff} & 6 & 3 & 2 & 11 & 12 & 3 \\
  \texttt{000b7} & 6 & 3 & 1 & 14 & 24 & 3 \\
  \texttt{0007f} & 6 & 3 & 2 & 13 & 12 & 3 \\
  \texttt{00037} & 6 & 3 & 1 & 15 & 4 & 3 \\
  \bottomrule
\end{longtable}

\end{center}

\clearpage

The remaining two tables list those matroids for which we could only compute either $\QAut{\pB}{\sM}$ or $\QAut{\pC}{\sM}$.
\begin{center}
\small
\begin{longtable}{lrrrrrr}
  \caption{$\QAut{\pB}{\sM}$ commutative but no data on $\QAut{\pC}{\sM}$} \label{Tab:computational-results-5} \\
  \toprule
  Matroid & $|E(\sM)|$ & $\rank(\sM)$ & $\girth(\sM)$ & \#nonbases & $|\Aut(\sM)|$ & $d_{\pB}(\sM)$ \\ \midrule
  \endfirsthead
  \caption{$\QAut{\pB}{\sM}$ commutative but no data on $\QAut{\pC}{\sM}$ (continued)} \\  
  Matroid & $|E(\sM)|$ & $\rank(\sM)$ & $\girth(\sM)$ & \#nonbases & $|\Aut(\sM)|$ & $d_{\pB}(\sM)$ \\ \midrule
  \endhead
  \endfoot
  \endlastfoot
  \texttt{f} & 4 & 3 & 4 & 0 & 24 & 3 \\
  \texttt{fffff} & 6 & 3 & 4 & 0 & 720 & 5 \\
  \texttt{7ffff} & 6 & 3 & 3 & 1 & 36 & 3 \\
  \texttt{7fffe} & 6 & 3 & 3 & 2 & 72 & 3 \\
  \texttt{7efff} & 6 & 3 & 3 & 2 & 8 & 3 \\
  \texttt{7efdf} & 6 & 3 & 3 & 3 & 6 & 3 \\
  \texttt{7efdd} & 6 & 3 & 3 & 4 & 24 & 3 \\
  \texttt{37dff} & 6 & 3 & 2 & 4 & 48 & 3 \\
  \texttt{379ff} & 6 & 3 & 2 & 5 & 12 & 3 \\
  \texttt{12cb7} & 6 & 3 & 1 & 10 & 120 & 4 \\
  \texttt{0fff7} & 6 & 3 & 3 & 5 & 12 & 3 \\
  \texttt{07dfd} & 6 & 3 & 2 & 7 & 4 & 3 \\
  \texttt{02cb7} & 6 & 3 & 1 & 11 & 12 & 3 \\
  \texttt{01c7f} & 6 & 3 & 2 & 10 & 36 & 3 \\
  \texttt{00c37} & 6 & 3 & 1 & 13 & 12 & 3 \\
  \texttt{0165b96ef} & 7 & 3 & 1 & 16 & 36 & 3 \\
  \texttt{0165b96ee} & 7 & 3 & 1 & 17 & 72 & 3 \\
  \texttt{0165996ef} & 7 & 3 & 1 & 17 & 8 & 3 \\
  \texttt{0165996af} & 7 & 3 & 1 & 18 & 6 & 3 \\
  \texttt{0165996ad} & 7 & 3 & 1 & 19 & 24 & 3 \\
  \texttt{0061b86ef} & 7 & 3 & 1 & 19 & 48 & 3 \\
  \texttt{0061b06ef} & 7 & 3 & 1 & 20 & 12 & 3 \\
  \texttt{0005b96e7} & 7 & 3 & 1 & 20 & 12 & 3 \\
  \texttt{0003f8ffd} & 7 & 3 & 2 & 17 & 12 & 3 \\
  \texttt{0001b86ed} & 7 & 3 & 1 & 22 & 4 & 3 \\
  \texttt{0000380ef} & 7 & 3 & 1 & 25 & 36 & 3 \\
  \texttt{0000380ee} & 7 & 3 & 1 & 26 & 36 & 3 \\
  \bottomrule
\end{longtable}

\end{center}

\begin{center}
\small
\begin{longtable}{lrrrrrr}
  \caption{$\QAut{\pB}{\sM}$ noncommutative but no data on $\QAut{\pC}{\sM}$} \label{Tab:computational-results-6} \\
  \toprule
  Matroid & $|E(\sM)|$ & $\rank(\sM)$ & $\girth(\sM)$ & \#nonbases & $|\Aut(\sM)|$ & $d_{\pB}(\sM)$ \\ \midrule
  \endfirsthead
  \caption{$\QAut{\pB}{\sM}$ noncommutative but no data on $\QAut{\pC}{\sM}$ (continued)} \\  
  Matroid & $|E(\sM)|$ & $\rank(\sM)$ & $\girth(\sM)$ & \#nonbases & $|\Aut(\sM)|$ & $d_{\pB}(\sM)$ \\ \midrule
  \endhead
  \endfoot
  \endlastfoot
  \texttt{0ffff} & 6 & 3 & 3 & 4 & 48 & 3 \\
  \texttt{07dff} & 6 & 3 & 2 & 6 & 8 & 3 \\
  \texttt{079ef} & 6 & 3 & 2 & 8 & 16 & 3 \\
  \texttt{7f} & 7 & 1 & 2 & 0 & 5040 & 3 \\
  \texttt{3f} & 7 & 1 & 1 & 1 & 720 & 3 \\
  \texttt{1f} & 7 & 1 & 1 & 2 & 240 & 3 \\
  \texttt{0f} & 7 & 1 & 1 & 3 & 144 & 3 \\
  \texttt{07} & 7 & 1 & 1 & 4 & 144 & 3 \\
  \texttt{03} & 7 & 1 & 1 & 5 & 240 & 3 \\
  \texttt{01} & 7 & 1 & 1 & 6 & 720 & 3 \\
  \texttt{1fffff} & 7 & 2 & 3 & 0 & 5040 & 3 \\
  \texttt{0fffff} & 7 & 2 & 2 & 1 & 240 & 3 \\
  \texttt{0f7fff} & 7 & 2 & 2 & 2 & 48 & 2 \\
  \texttt{0f7fbf} & 7 & 2 & 2 & 3 & 48 & 3 \\
  \texttt{05bbdf} & 7 & 2 & 1 & 6 & 720 & 3 \\
  \texttt{03ffff} & 7 & 2 & 2 & 3 & 144 & 3 \\
  \texttt{03f7ff} & 7 & 2 & 2 & 4 & 24 & 2 \\
  \texttt{03f7fe} & 7 & 2 & 2 & 5 & 48 & 2 \\
  \texttt{03f73f} & 7 & 2 & 2 & 6 & 72 & 2 \\
  \texttt{01bbdf} & 7 & 2 & 1 & 7 & 48 & 3 \\
  \texttt{01b3df} & 7 & 2 & 1 & 8 & 16 & 2 \\
  \texttt{01b3de} & 7 & 2 & 1 & 9 & 48 & 3 \\
  \texttt{007fff} & 7 & 2 & 2 & 6 & 144 & 3 \\
  \texttt{007fbf} & 7 & 2 & 2 & 7 & 48 & 3 \\
  \texttt{007fbc} & 7 & 2 & 2 & 9 & 144 & 3 \\
  \texttt{003bdf} & 7 & 2 & 1 & 9 & 36 & 2 \\
  \texttt{003b9f} & 7 & 2 & 1 & 10 & 12 & 2 \\
  \texttt{003b9c} & 7 & 2 & 1 & 12 & 72 & 2 \\
  \texttt{0099cf} & 7 & 2 & 1 & 11 & 240 & 3 \\
  \texttt{0007ff} & 7 & 2 & 2 & 10 & 240 & 3 \\
  \texttt{0007fe} & 7 & 2 & 2 & 11 & 240 & 3 \\
  \texttt{0003df} & 7 & 2 & 1 & 12 & 48 & 3 \\
  \texttt{0003de} & 7 & 2 & 1 & 13 & 48 & 3 \\
  \texttt{0019cf} & 7 & 2 & 1 & 12 & 24 & 2 \\
  \texttt{00198f} & 7 & 2 & 1 & 13 & 16 & 2 \\
  \texttt{0008c7} & 7 & 2 & 1 & 15 & 144 & 3 \\
  \texttt{00003f} & 7 & 2 & 2 & 15 & 720 & 3 \\
  \texttt{0001cf} & 7 & 2 & 1 & 14 & 24 & 2 \\
  \texttt{0001ce} & 7 & 2 & 1 & 15 & 24 & 2 \\
  \texttt{0000c7} & 7 & 2 & 1 & 16 & 24 & 2 \\
  \texttt{0000c6} & 7 & 2 & 1 & 17 & 48 & 2 \\
  \texttt{000043} & 7 & 2 & 1 & 18 & 144 & 3 \\
  \texttt{00001f} & 7 & 2 & 1 & 16 & 120 & 3 \\
  \texttt{00000f} & 7 & 2 & 1 & 17 & 48 & 3 \\
  \texttt{000007} & 7 & 2 & 1 & 18 & 36 & 2 \\
  \texttt{000003} & 7 & 2 & 1 & 19 & 48 & 3 \\
  \texttt{000001} & 7 & 2 & 1 & 20 & 240 & 3 \\
  \texttt{07fffffff} & 7 & 3 & 3 & 4 & 144 & 3 \\
  \texttt{00e3f0fdc} & 7 & 3 & 2 & 17 & 144 & 3 \\
  \texttt{002098267} & 7 & 3 & 1 & 25 & 240 & 4 \\
  \texttt{000ffbfe0} & 7 & 3 & 2 & 17 & 24 & 3 \\
  \texttt{000f7bde7} & 7 & 3 & 2 & 16 & 8 & 3 \\
  \texttt{000f7bde0} & 7 & 3 & 2 & 19 & 16 & 3 \\
  \texttt{0005b96ef} & 7 & 3 & 1 & 19 & 48 & 3 \\
  \texttt{0005b96e0} & 7 & 3 & 1 & 23 & 48 & 3 \\
  \texttt{0003f8fff} & 7 & 3 & 2 & 16 & 24 & 3 \\
  \texttt{0003f8fe3} & 7 & 3 & 2 & 19 & 48 & 5 \\
  \texttt{0003f0fdf} & 7 & 3 & 2 & 18 & 24 & 3 \\
  \texttt{0003f0fdc} & 7 & 3 & 2 & 20 & 24 & 3 \\
  \texttt{0003f0fc3} & 7 & 3 & 2 & 21 & 24 & 3 \\
  \texttt{0003f0fc0} & 7 & 3 & 2 & 23 & 48 & 4 \\
  \texttt{0001b86ef} & 7 & 3 & 1 & 21 & 8 & 3 \\
  \texttt{0001b86e3} & 7 & 3 & 1 & 23 & 16 & 3 \\
  \texttt{0001b06cf} & 7 & 3 & 1 & 23 & 16 & 3 \\
  \texttt{0001b06c3} & 7 & 3 & 1 & 25 & 8 & 3 \\
  \texttt{0001b06c0} & 7 & 3 & 1 & 27 & 48 & 3 \\
  \texttt{000098267} & 7 & 3 & 1 & 26 & 24 & 3 \\
  \texttt{000098263} & 7 & 3 & 1 & 27 & 16 & 3 \\
  \texttt{0000781ff} & 7 & 3 & 2 & 22 & 144 & 3 \\
  \texttt{0000781fe} & 7 & 3 & 2 & 23 & 144 & 3 \\
  \texttt{000018067} & 7 & 3 & 1 & 28 & 24 & 3 \\
  \texttt{000018066} & 7 & 3 & 1 & 29 & 24 & 3 \\
  \texttt{000008023} & 7 & 3 & 1 & 31 & 144 & 3 \\
  \texttt{000007fff} & 7 & 3 & 3 & 20 & 720 & 3 \\
  \texttt{000003fff} & 7 & 3 & 2 & 21 & 48 & 3 \\
  \texttt{000003dff} & 7 & 3 & 2 & 22 & 16 & 3 \\
  \texttt{000003dfe} & 7 & 3 & 2 & 23 & 48 & 3 \\
  \texttt{0000016ef} & 7 & 3 & 1 & 25 & 120 & 3 \\
  \texttt{000000fff} & 7 & 3 & 2 & 23 & 36 & 3 \\
  \texttt{000000fdf} & 7 & 3 & 2 & 24 & 12 & 3 \\
  \texttt{000000fdc} & 7 & 3 & 2 & 26 & 72 & 3 \\
  \texttt{0000006ef} & 7 & 3 & 1 & 26 & 12 & 3 \\
  \texttt{0000006cf} & 7 & 3 & 1 & 27 & 8 & 3 \\
  \texttt{000000267} & 7 & 3 & 1 & 29 & 48 & 3 \\
  \texttt{0000001ff} & 7 & 3 & 2 & 26 & 48 & 3 \\
  \texttt{0000001fe} & 7 & 3 & 2 & 27 & 48 & 3 \\
  \texttt{0000000ef} & 7 & 3 & 1 & 28 & 12 & 3 \\
  \texttt{0000000ee} & 7 & 3 & 1 & 29 & 12 & 3 \\
  \texttt{000000067} & 7 & 3 & 1 & 30 & 8 & 3 \\
  \texttt{000000066} & 7 & 3 & 1 & 31 & 16 & 3 \\
  \texttt{000000023} & 7 & 3 & 1 & 32 & 36 & 3 \\
  \texttt{00000001f} & 7 & 3 & 2 & 30 & 240 & 3 \\
  \texttt{00000000f} & 7 & 3 & 1 & 31 & 48 & 3 \\
  \texttt{000000007} & 7 & 3 & 1 & 32 & 24 & 3 \\
  \texttt{000000003} & 7 & 3 & 1 & 33 & 24 & 3 \\
  \texttt{000000001} & 7 & 3 & 1 & 34 & 144 & 3 \\
  \bottomrule
\end{longtable}

\end{center}

\end{document}